\documentclass[twoside,11pt,reqno]{amsart}
  \usepackage{amsmath}
  \usepackage{amssymb}
  \usepackage{amscd}
  \usepackage{mathrsfs}
  \usepackage{epic}
  \usepackage{latexsym}
  \usepackage{tikz}
  \usepackage{mathrsfs}
  \usepackage{cite}
  \usepackage{hyperref}
  \usepackage{mathtools}
  \usepackage{pb-diagram}
  \usepackage{tikz-cd}
  \usepackage[all]{xy}
  \usepackage{xcolor}
  \usepackage{enumitem}
  \usepackage{xargs}
  \usepackage{xparse}
  \usepackage{tensor}
  \usepackage{xpatch}
  \usepackage[nameinlink, noabbrev,capitalize]{cleveref}
  \usepackage{anyfontsize}

  \usetikzlibrary{calc}
  \usetikzlibrary{cd}
  \usetikzlibrary{arrows.meta}
  \usetikzlibrary{positioning}

  \newcommand{\itemCref}[2]{\hyperref[{#2}]{\mbox{\Cref*{#1}\labelcref*{#2}}}}
  \newcommand{\PartCref}[1]{\hyperref[{#1}]{Part \labelcref*{#1}}}
  \newcommand{\partCref}[1]{\hyperref[{#1}]{part \labelcref*{#1}}}
  \newcommand{\caseCref}[1]{\hyperref[{#1}]{case \labelcref*{#1}}}
  \newcommand{\CaseCref}[1]{\hyperref[{#1}]{Case \labelcref*{#1}}}
  \newcommand{\secCref}[1]{\hyperref[{#1}]{§\labelcref*{#1}}}
  \newcommand{\SecCref}[1]{\hyperref[{#1}]{(§\labelcref*{#1})}}
  \NewDocumentCommand{\myCref}{mo}{\IfNoValueTF{#2}{\Cref{#1}}{\itemCref{#1}{#2}}}

  \allowdisplaybreaks

  \numberwithin{equation}{section}
  \newtheorem{Proposition}[equation]{Proposition}
  \newtheorem{Lemma}[equation]{Lemma}
  \newtheorem{Theorem}[equation]{Theorem}
  \newtheorem{Corollary}[equation]{Corollary}
  \theoremstyle{definition} 
  \newtheorem{Definition}[equation]{Definition}
  \newtheorem{Remark}[equation]{Remark}
  
  \newtheorem{Example}[equation]{Example}
  \newtheorem{Examples}[equation]{Examples}
  
  \Crefname{Proposition}{Proposition}{Propositions}
  \crefname{Proposition}{proposition}{propositions}
  \Crefname{Examples}{Examples}{Examples}
  \crefname{examples}{examples}{examples}
  \Crefname{Example}{Example}{Examples}
  \crefname{example}{example}{examples}

  \newcommand{\settext}[1]{\text{\normalfont\ #1}\ } 
  \newcommand{\bmath}[1]{\text{\boldmath$ #1$}} 

  \let\ss\scriptstyle
  \let\sss\scriptscriptstyle
  \newcommand{\mtiny}[1]{\mbox{\tiny$ #1$}}
  \newcommand{\mtinier}[1]{\mbox{\tinier$ #1$}}
  \newcommand{\mminiscule}[1]{ \mbox{\miniscule$ #1$}}

  \usepackage{lmodern}
  \makeatletter
  \ifcase \@ptsize \relax
  \newcommand{\miniscule}{\@setfontsize\miniscule{4}{5}}
  \or
  \newcommand{\miniscule}{\@setfontsize\miniscule{4}{5}}
  \or
  \newcommand{\miniscule}{\@setfontsize\miniscule{5}{6}}
  \fi
  \makeatother
  \makeatletter
  \ifcase \@ptsize \relax
  \newcommand{\tinier}{\@setfontsize\tinier{4}{5}}
  \or
  \newcommand{\tinier}{\@setfontsize\tinier{5}{6}}
  \or
  \newcommand{\tinier}{\@setfontsize\tinier{5}{6}}
  \fi
  \makeatother

  \NewDocumentCommand\smallsup{mom}{\IfNoValueTF{#2}{{{{#1}^{{\sss{#3}}}}}\vphantom{#1}}{{{{#1}_{#2}^{{\sss{#3}}}}}\vphantom{#1}}}
  \NewDocumentCommand\tinysup{mom}{\IfNoValueTF{#2}{{{{#1}^{{\mtiny{#3}}}}}\vphantom{#1}}{{{{#1}_{#2}^{{\mtiny{#3}}}}}\vphantom{#1}}}
  \NewDocumentCommand\tiniersup{mom}{\IfNoValueTF{#2}{{{{#1}^{{\mtinier{#3}}}}}\vphantom{#1}}{{{{#1}_{#2}^{{\mtinier{#3}}}}}\vphantom{#1}}}
  \NewDocumentCommand\tiniestsup{mom}{\IfNoValueTF{#2}{{{{#1}^{{\mminiscule{#3}}}}}\vphantom{#1}}{{{{#1}_{#2}^{{\mminiscule{#3}}}}}\vphantom{#1}}}
  \NewDocumentCommand\lowsmallsup{mom}{\IfNoValueTF{#2}{\def\arraystretch{0.2}#1\begin{array}{@{}l@{}}{\sss #3}\\ {\phantom{\sss #1}}\end{array}\vphantom{#1}}{\def\arraystretch{0.1}#1\begin{array}{@{}l@{}}{\sss #3}\\ {\vphantom{#1}^{\ss #2}}\end{array}\vphantom{#1}}}
  \NewDocumentCommand\lowsmallestsup{mom}{\IfNoValueTF{#2}{\def\arraystretch{0.1}#1\begin{array}{@{}l@{}}{\mminiscule{#3}}\\ {\phantom{\mminiscule{#1}}}\end{array}\vphantom{#1}}{\def\arraystretch{0.1}#1\begin{array}{@{}l@{}}{\mminiscule{#3}}\\ {\vphantom{#1}^{\ss #2}}\end{array}\vphantom{#1}}}
  \NewDocumentCommand\lowsup{mom}{\IfNoValueTF{#2}{\def\arraystretch{0.1}#1\begin{array}{@{}l@{}}{ \ss{#3}}\\ {\phantom{\ss{#1}}}\end{array}\vphantom{#1}}{\def\arraystretch{0.1}#1\begin{array}{@{}l@{}}{\ss{#3}}\\ {\vphantom{#1}^{\ss #2}}\end{array}\vphantom{#1}}}

  \NewDocumentCommand\smallsub{mmo}{\IfNoValueTF{#3}{{{{#1}_{{\sss{#2}}}}}\vphantom{#1}}{{{{#1}^{#3}_{{\sss{#2}}}}}\vphantom{#1}}}
  \NewDocumentCommand\tinysub{mmo}{\IfNoValueTF{#3}{{{{#1}_{{\mtiny{#2}}}}}\vphantom{#1}}{{{{#1}^{#3}_{{\mtiny{#2}}}}}\vphantom{#1}}}
  \NewDocumentCommand\tiniersub{mmo}{\IfNoValueTF{#3}{{{{#1}_{{\mtinier{#2}}}}}\vphantom{#1}}{{{{#1}^{#3}_{{\mtinier{#2}}}}}\vphantom{#1}}}
  \NewDocumentCommand\tiniestsub{mom}{\IfNoValueTF{#3}{{{{#1}_{{\mminiscule{#2}}}}}\vphantom{#1}}{{{{#1}^{#3}^{{\mminiscule{#2}}}}}\vphantom{#1}}}
  \NewDocumentCommand\lowsmallsub{mmo}{\IfNoValueTF{#3}{\def\arraystretch{0.3}#1\begin{array}{@{}l@{}}{\phantom{\sss #1}}\\ {\sss #2}\end{array}\vphantom{#1}}{\def\arraystretch{0.1}#1\begin{array}{@{}l@{}}{{\ss #3}}\\ {\sss #2}\end{array}\vphantom{#1}}}
  \NewDocumentCommand\lowsmallestsub{mmo}{\IfNoValueTF{#3}{\def\arraystretch{0.1}#1\hspace{-0.2ex}\begin{array}{@{}l@{}}{\phantom{\mminiscule #1}}\\ {\mminiscule #2}\end{array}\vphantom{#1}\hspace{-0.4ex}}{\def\arraystretch{0.1}#1\begin{array}{@{}l@{}}{{\ss #3}}\\ {\mminiscule #2}\end{array}\vphantom{#1}}}

  \newcommand{\smallsubsup}[3]{{\def\arraystretch{0.5}#1\begin{array}{@{}l@{}}{\sss #3}\\ {\sss #2}\end{array}}}

  \newcommand{\modules}{\textrm{-mod}}
  \renewcommand{\mod}[1]{\ (\operatorname{mod} #1)}
  \newcommand{\coker}{{\rm{coker}}\,}
  \newcommand{\Sp}{{\mathrm {Sp}}}
  \newcommand{\Hom}{\operatorname{Hom}}
  
  \newcommand{\End}{\operatorname{End}}
  \newcommand{\GL}{\operatorname{GL}}
  \newcommand{\Char}{{\mathrm {char}\ }}
  \newcommand{\cf}{\text{\rm{cf}}}
  \newcommand{\diag}{{\mathrm {diag}}}
  \newcommand{\ind}{\operatorname{ind}}
  \newcommand{\sgn}{\mathtt{sgn}}
  \newcommand{\im}{{\mathrm{im}\,}}
  
  \newcommand{\Tab}{{\rm{Tab}}}
  \newcommand{\Rel}{\operatorname{Rel}}
  \newcommand{\SRel}{\operatorname{SRel}}

  \newcommand{\rowexop}[5]{\smallsubsup{#1}{(#3,#4)(#2,#5)}{(#2,#4)(#3,#5)}}
  \newcommand{\colexop}[5]{\smallsubsup{#1}{(#4,#3)(#5,#2)}{(#4,#2)(#5,#3)}}

  \DeclareFontFamily{U}{mathx}{\hyphenchar\font45}
  \DeclareFontShape{U}{mathx}{m}{n}{
  <5> <6> <7> <8> <9> <10>
  <10.95> <12> <14.4> <17.28> <20.74> <24.88>
  mathx10
  }{}
  \DeclareSymbolFont{mathx}{U}{mathx}{m}{n}
  \DeclareFontSubstitution{U}{mathx}{m}{n}
  
  \let\bigoplus\relax
  \DeclareMathSymbol{\bigoplus}{1}{mathx}{"C0}

  \newcommand{\flip}[1]{\reflectbox{$ #1$}}
  \newcommand{\eqqcolon}{\mathrel{\flip{\coloneqq}}}

  \newcommand{\sq}[1]{{\left[#1\right]}}
  
  \def\k{\Bbbk}

  \newcommand{\xr}[1]{\xrightarrow{#1}}
  
  \newcommand{\xlr}[1]{\xleftrightarrow{#1}}
  \newcommandx{\xrtail}[1][1 = \hspace{0.1 ex}]{\overset{#1}{\rightarrowtail}}
  \newcommandx{\totail}[1][1 = \hspace{0.1 ex}]{\overset{#1}{\rightarrowtail}}
  \newcommandx{\xrtwo}[1][1 = \hspace{0.1 ex}]{\overset{#1}{\twoheadrightarrow}}
  
  \newcommandx{\into}[1][1 = \hspace{0.1 ex}]{\xhookrightarrow{#1}}
  \renewcommandx{\to}[1][1 = \hspace{0.1ex}]{\xr{#1}}
  \newcommandx{\leftright}[1][1 = \hpsace{0.1ex}]{\xlr{#1}}
  \newcommandx{\from}[1][1 = \hpsace{0.1ex}]{\xr{#1}}
  \newcommandx{\eqto}[1][1 = \hspace{2 ex}]{\xlongequal{#1}}
  \let\oldmapsto\mapsto
  \RenewDocumentCommand\mapsto{o}{\IfNoValueTF{#1}{\oldmapsto}{\xmapsto{#1}}}
  
  \newcommand{\xrightarrowdbl}[2][]{%
  \xrightarrow[#1]{#2}\mathrel{\mkern-14mu}\rightarrow
  }
  \newcommandx{\onto}[2][1 = \hspace{0.1 ex},2 = \hspace{0.1 ex},]{\xrightarrowdbl[#2]{#1}}
  \makeatletter
  \newcommand*{\da@rightarrow}{\mathchar"0\hexnumber@\symAMSa 4B }
  \newcommand*{\da@leftarrow}{\mathchar"0\hexnumber@\symAMSa 4C }
  \newcommand*{\xdashrightarrow}[2][]{%
  \mathrel{%
    \mathpalette{\da@xarrow{#1}{#2}{}\da@rightarrow{\,}{}}{}%
  }%
  }
  \newcommand{\xdashleftarrow}[2][]{%
  \mathrel{%
    \mathpalette{\da@xarrow{#1}{#2}\da@leftarrow{}{}{\,}}{}%
  }%
  }
  \newcommand*{\da@xarrow}[7]{%
  \sbox0{$\ifx#7\scriptstyle\scriptscriptstyle\else\scriptstyle\fi#5#1#6\m@th$}%
  \sbox2{$\ifx#7\scriptstyle\scriptscriptstyle\else\scriptstyle\fi#5#2#6\m@th$}%
  \sbox4{$ #7\dabar@\m@th$}%
  \dimen@=\wd0 %
  \ifdim\wd2 >\dimen@
    \dimen@=\wd2 %
  \fi
  \count@=2 %
  \def\da@bars{\dabar@\dabar@}%
  \@whiledim\count@\wd4<\dimen@\do{%
    \advance\count@\@ne
    \expandafter\def\expandafter\da@bars\expandafter{%
      \da@bars
      \dabar@
    }%
  }%
  \mathrel{#3}%
  \mathrel{%
    \mathop{\da@bars}\limits
    \ifx\\#1\\%
    \else
      _{\copy0}%
    \fi
    \ifx\\#2\\%
    \else
      ^{\copy2}%
    \fi
  }%
  \mathrel{#4}%
  }
  \newcommandx{\dashto}[1][1 = \hspace{0.1 ex}]{\xdashrightarrow{#1}}
  \renewcommandx{\mapsto}[1][1 = \hspace{0.1ex}]{\xmapsto{#1}}

  \RenewDocumentCommand{\brace}{mo}{\IfNoValueTF{#2}{\left\{ #1 \right\}}{\left\{ #1 \; \middle| \; #2 \right\}}}

\begin{document}

  \date{\today}

  \thanks{The authors gratefully acknowledge the support of The Royal Society through the research grant RGF\textbackslash R1\textbackslash181015a.}

  \title[Endomorphisms of Specht modules]{{\bf On the endomorphism algebra of Specht modules in even characteristic}}

  \author{\sc Haralampos Geranios}
  \address{Department of Mathematics\\ University of York\\ York YO10 5DD, U.K.}
  \email{haralampos.geranios@york.ac.uk}

  \author{\sc Adam Higgins}
  \address{Department of Mathematics\\ University of York\\ York YO10 5DD, U.K.}
  \email{adam.higgins@york.ac.uk}

  \begin{abstract}
    Over fields of characteristic $2$, Specht modules may decompose and there is no upper bound for the dimension of their endomorphism algebra. A classification of the (in)decomposable Specht modules and a closed formula for the dimension of their endomorphism algebra remain two important open problems in the area. In this paper, we introduce a novel description of the endomorphism algebra of the Specht modules and provide infinite families of Specht modules with one-dimensional endomorphism algebra.
  \end{abstract}

  \maketitle

\section{Introduction}\label{sec.intro}

Let $\k$ be an algebraically closed field of characteristic $p\geq 0$ and $r$ a positive integer. We write $\mathfrak{S}_{r}$ for the symmetric group on $r$ letters and $\k\mathfrak{S}_{r}$ for its group algebra over $\k$. For each partition $\lambda$ of $r$ we have the Specht module $\Sp(\lambda)$ and for each composition $\alpha$ of $r$ we have the permutation module $M(\alpha)$. Recall that $\Sp(\lambda)$ may be viewed as a submodule of $M(\lambda)$. One fundamental result by James states that unless the characteristic of $\k$ is $2$ and $\lambda$ is \mbox{$2$-singular}, the space of homomorphisms $\Hom_{\k\mathfrak{S}_{r}}(\Sp(\lambda),M(\lambda))$ is one-dimensional\mbox{\cite[Corollary 13.17]{J}}. It follows that the endomorphism algebra of $\Sp(\lambda)$ is one-dimensional and so in particular that $\Sp(\lambda)$ is indecomposable.

In contrast, if the characteristic of $\k$ is $2$ and $\lambda$ is a \mbox{$2$-singular} partition, that is $\lambda$ has a repeated term, $\Sp(\lambda)$ may certainly decompose. The first example of a decomposable Specht module was discovered by James in the late 70s, thereby setting in motion the investigation of the (in)decomposability of Specht modules; a problem that has attracted a lot of attention over the years. In a recent paper \cite{DG1}, Donkin and the first author considered partitions of the form $\lambda=(a,m-1,m-2,\ldots,2,1^{b})$ and obtained precise decompositions of $\Sp(\lambda)$ in the case where $a-m$ is even and $b$ is odd. An interesting feature arising in these decompositions is that there is no upper bound for the number of indecomposable summands of $\Sp(\lambda)$ and so in turn for the dimension of its endomorphism algebra \cite[Example 6.3]{DG1}. Almost half a century after James' first example, a classification of the (in)decomposable Specht modules remains to be found and there is no known formula describing the dimension of their endomorphism algebra. In this paper, we provide a new characterisation of $\End_{\k\mathfrak{S}_{r}}(\Sp(\lambda))$ as a subset of the homomorphism space $\Hom_{\k\mathfrak{S}_{r}}(M(\lambda'),M(\lambda))$, where $\lambda'$ is the transpose partition of $\lambda$. Our description allows one to realise an endomorphism of $\Sp(\lambda)$ as an element of the set $\Hom_{\k\mathfrak{S}_{r}}(M(\lambda'),M(\lambda))$ that satisfies certain concrete relations. In this way, we are able to show that for $\lambda=(a,m-1,\ldots,2,1^{b})$ with $a-m\equiv b\mod{2}$, the endomorphism algebra of $\Sp(\lambda)$ is one-dimensional.

We do so by taking inspiration from the category of polynomial representations of the general linear groups. More precisely, for a partition $\lambda$, we compare two different constructions of the induced module $\nabla(\lambda)$ for $\GL_{n}(\k)$: the first introduced by Akin, Buchsbaum, and Weyman \cite[Theorem II.2.11]{ABW} and the second by James\mbox{\cite[Theorem 26.3(ii)]{J}}. By applying the Schur functor \cite[§6.3]{G}, we then obtain two characterisations of the Specht module $\Sp(\lambda)$: first as a quotient of $M(\lambda')$ and then as a submodule of $M(\lambda)$. This leads to a concrete description of the endomorphism algebra of $\Sp(\lambda)$, which we shall then investigate in detail for partitions of the form $\lambda=(a,m-1,\dots,2,1^{b})$.

The paper is arranged in the following way. \Cref{sec.prel} provides the necessary background on polynomial representations of $\GL_{n}(\k)$ and \mbox{$\k\mathfrak{S}_{r}$-modules}. In \Cref{sec.end.alg} we explore the connection between these two categories via the Schur functor $f$ and its right-inverse $g$. As a by-product of our considerations, we provide a new short proof of the fact that $g\Sp(\lambda)\cong\nabla(\lambda)$ for $p\neq 2$. Then, we focus on homomorphisms and in \Cref{end.alg.Sp} we obtain the desired description of $\End_{\k\mathfrak{S}_{r}}(\Sp(\lambda))$ in characteristic $2$. In \Cref{sec.trunc.trick} we utilise more tools from the representation theory of $\GL_{n}(\k)$ to obtain a reduction technique that will be instrumental to our investigation of the case $\lambda=(a,m-1,\ldots,2,1^{b})$ in \Cref{sec.end.alg.main}.

\section{Preliminaries}\label{sec.prel}

We write $\mathbb{N}$ for the set of non-negative integers.

\subsection{Combinatorics}\label{subsec.comb}

Let $\ell$ be a positive integer and $\alpha=(\alpha_{1},\ldots,\alpha_{\ell})$ be an $\ell$-tuple of non-negative integers. We let $\deg(\alpha)\coloneqq\alpha_{1}+\cdots+\alpha_{\ell}$ and call it the {\emph{degree}} of $\alpha$. We define the \emph{length} of $\alpha$, denoted $\ell(\alpha)$, to be the maximal positive integer $l$ with $1\leq l\leq\ell$ such that $\alpha_{l}\neq 0$ if $\alpha$ is non-zero, and we set $\ell(\alpha)\coloneqq 0$ for $\alpha=(0^{\ell})$. Now, fix positive integers $n$ and $r$. We write $\Lambda(n)$ for the set of $n$-tuples of non-negative integers, and $\Lambda^{+}(n)$ for the set of partitions with at most $n$ parts. We write $\Lambda(n,r)$ for the subset of $\Lambda(n)$ consisting of those elements of degree $r$, and $\Lambda^{+}(n,r)$ for the partitions of $r$ with at most $n$ parts. Given a partition $\lambda\in\Lambda^{+}(n)$, we write $\lambda'$ for its transpose partition. For $\alpha\in\Lambda(n)$ and $1\leq i<j\leq\ell(\alpha)$ with $\alpha_{j}\neq 0$, and for $0<k\leq\alpha_{j}$, we shall denote by $\smallsup{\alpha}{(i,j,k)}=(\smallsup{\alpha}[1]{(i,j,k)},\smallsup{\alpha}[2]{(i,j,k)},\ldots)$ the element of $\Lambda(n)$ with terms $\smallsup{\alpha}[l]{(i,j,k)}\coloneqq\alpha_{l}+k(\delta_{i,l}-\delta_{j,l})$.

\subsection{Representations of general linear groups}\label{subsec.pol.rep}

We consider the general linear group $G\coloneqq\GL_{n}(\k)$ and its coordinate algebra $\k[G]=\k[c_{11},\ldots,c_{nn},\det^{-1}]$, where $\det$ is the determinant function. We write $A_{\k}(n)\coloneqq\k[c_{11},\ldots,c_{nn}]$ for the polynomial subalgebra of $\k[G]$ generated by the functions $c_{ij}$ with $1\leq i,j\leq n$. The algebra $A_{\k}(n)$ has an \mbox{$\mathbb{N}$-grading} of the form $A_{\k}(n)=\bigoplus_{r\in\mathbb{N}}A_{\k}(n,r)$ where $A_{\k}(n,r)$ consists of the homogeneous degree $r$ polynomials in the $c_{ij}$. Given a rational \mbox{$G$-module} $V$, we shall denote by $\cf(V)$ the \emph{coefficient space} of $V$, that is the subspace of $\k[G]$ generated by the \emph{coefficient functions} \mbox{$f_{vv'}:G\to\k$} satisfying \mbox{$g\cdot v'=\sum_{v\in\mathcal{V}}f_{vv'}(g)v$} for $g\in G$, $v,v'\in\mathcal{V}$, where $\mathcal{V}$ is some \mbox{$\k$-basis} of $V$. We say that $V$ is a {\emph{polynomial representation}} of $G$ if $\cf(V)\subseteq A_{\k}(n)$ and a \emph{polynomial representation of $G$ of degree $r$} if $\cf(V)\subseteq A_{\k}(n,r)$. We write $M_{\k}(n)$ for the category of polynomial representations of $G$ and $M_{\k}(n,r)$ for its subcategory of representations of degree $r$. Recall that the category $M_{\k}(n,r)$ is naturally equivalent to the category of $S_{\k}(n,r)$-modules, where $S_{\k}(n,r)\coloneqq A_{\k}(n,r)^{*}$ is the corresponding \mbox{\emph{Schur algebra}}\mbox{\cite[§2.3, §2.4]{G}}. For $V\in M_{\k}(n)$ we write $V^{\circ}$ for its contravariant dual, in the sense of \cite[§2.7]{G}.

We fix $T$ to be the maximal torus of $G$ consisting of the diagonal matrices in $G$. An element $\alpha\in\Lambda(n)$ may be identified with the multiplicative character of $T$ that takes an element $t=\diag(t_{1},\ldots,t_{n})\in T$ to $\alpha(t)\coloneqq t_{1}^{\alpha_{1}}\cdots t_{n}^{\alpha_{n}}\in\k$. We denote by $\k_{\alpha}$ the one-dimensional rational \mbox{$T$-module} on which $t\in T$ acts by multiplication by $\alpha(t)$. Then, given $V\in M_{\k}(n)$, $\alpha\in\Lambda(n)$, we write $V^{\alpha}\coloneqq\{v\in V\mid t\cdot v=\alpha(t)v\settext{for all}t\in T\}$ for the \mbox{$\alpha$-{\emph{weight space}}} of $V$. We write $E\coloneqq\k^{\oplus n}$ for the natural \mbox{$G$-module} and $S^{r}E$ (resp. $\Lambda^{r}E$, $D^{r}E$) for the corresponding $r$th symmetric power (resp. exterior power, divided power) of $E$. For $\ell\geq 1$ and an $\ell$-tuple $\alpha=(\alpha_{1},\ldots,\alpha_{\ell})$ of non-negative integers, we define the polynomial \mbox{$G$-modules}: \mbox{$S^{\alpha}E\coloneqq S^{\alpha_{1}}E\otimes\cdots\otimes S^{\alpha_{\ell}}E$}, \mbox{$\Lambda^{\alpha}E\coloneqq\Lambda^{\alpha_{1}}E\otimes\cdots\otimes\Lambda^{\alpha_{\ell}}E$}, and \mbox{$D^{\alpha}E\coloneqq D^{\alpha_{1}}E\otimes\cdots\otimes D^{\alpha_{\ell}}E$}. If $\deg(\alpha)=r$, then each of these modules lies in $M_{\k}(n,r)$. For $V\in M_{\k}(n)$, there is a \mbox{$\k$-linear} isomorphism $\Hom_{G}(V,S^{\alpha}E)\cong V^{\alpha}$\mbox{\cite[§2.1(8)]{D2}}. For $\alpha\in\Lambda(n)$, the \mbox{$T$-action} on $\k_{\alpha}$ extends uniquely to a module action of the subgroup $B\subseteq G$ of lower-triangular matrices. For $\lambda\in\Lambda^{+}(n)$, we write $\nabla(\lambda)\coloneqq\ind_{B}^{G}\k_{\lambda}$ for the \emph{induced \mbox{$G$-module}} corresponding to $\lambda$\mbox{\cite[§II.2]{Jan}}. Recall that there is a \mbox{$G$-isomorphism} $\nabla(\lambda)^{\circ}\cong\Delta(\lambda)$, where $\Delta(\lambda)$ is the \mbox{\emph{Weyl module}} corresponding to $\lambda$\mbox{\cite[§II.2.13(1)]{Jan}}.

Here, we shall review a construction of the induced module by Akin, Buchsbaum, and Weyman. In \cite[§II.1]{ABW}, the authors associate to a partition $\lambda$ with $\lambda_{1}\leq n$, a $G$-module denoted $L_{\lambda}(E)$, which they call the \emph{Schur functor of $E$}. For simplicity, we shall refer to the module $L_{\lambda}(E)$ as the \emph{Schur module associated to $\lambda$}. Further, in \cite[§II.2]{ABW} the authors provide a description of $L_{\lambda}(E)$ by generators and relations. More precisely, in \cite[Theorem II.2.16]{ABW}, the authors identify $L_{\lambda}(E)$ with the cokernel of a $G$-homomorphism between a pair of (direct sums of) tensor products of exterior powers of $E$. By \cite[§2.7(5)]{D3}, we have that $L_{\lambda}(E)$ is isomorphic to an induced module, namely $L_{\lambda}(E)\cong\nabla(\lambda')$ for $\lambda\in\Lambda^{+}(n)$ (note that $Y(\lambda)$ is used in place of $\nabla(\lambda)$ in \cite{D3}). Their construction is as follows. Recall that the exterior algebra $\Lambda(E)$ of $E$ enjoys a Hopf algebra structure \cite[§I.2]{ABW}. We write $\Delta$ and $\mu$ for the comultiplication and multiplication of $\Lambda(E)$ respectively. Let $\lambda$ be a partition with $\ell\coloneqq\ell(\lambda)$. For $1\leq i<\ell$, $1<j\leq\ell$, $t\geq 1$, and $1\leq s\leq\lambda_{j}$, we consider the \mbox{$G$-homomorphisms} $\smallsup{\Delta}[\lambda]{(i,t)}:\Lambda^{\lambda_{i}+t}E\to\Lambda^{\lambda_{i}}E\otimes\Lambda^{t}E$ and $\smallsup{\mu}[\lambda]{(j,s)}:\Lambda^{s}E\otimes\Lambda^{\lambda_{j}-s}E\to\Lambda^{\lambda_{j}}E$, coming from $\Delta$ and $\mu$ respectively. Further, for $1\leq i<j\leq\ell$, $1\leq s\leq\lambda_{j}$ we construct the \mbox{$G$-homomorphism} \mbox{$\smallsup{\phi}[\lambda]{(i,j,s)}:\Lambda^{\smallsup{\lambda}{(i,j,s)}}E\to\Lambda^{\lambda}E$} as the composition:
  \begin{equation}
    \begin{aligned}\label{General.Buch.I}
      &\Lambda^{\smallsup{\lambda}{(i,j,s)}}E\to[{1\otimes\cdots\otimes\smallsup{\Delta}[\lambda]{(i,s)}\otimes\cdots\otimes 1}]\Lambda^{\lambda_{1}}E\otimes\cdots\otimes\Lambda^{\lambda_{i}}E\otimes\Lambda^{s}E\otimes\cdots\otimes\Lambda^{\lambda_{j}-s}E\otimes\cdots\otimes\Lambda^{\lambda_{\ell}}E \\
      &\to[{\sigma}]\Lambda^{\lambda_{1}}E\otimes\cdots\otimes\Lambda^{\lambda_{i}}E\otimes\cdots\otimes\Lambda^{s}E\otimes\Lambda^{\lambda_{j}-s}E\otimes\cdots\otimes\Lambda^{\lambda_{\ell}}E\to[{1\otimes\cdots\otimes\smallsup{\mu}[\lambda]{(j,s)}\otimes\cdots\otimes 1}]\Lambda^{\lambda}E,
    \end{aligned}
  \end{equation}
where $\sigma$ denotes the isomorphism that permutes the corresponding tensor factors, and each $1$ refers to the identity map on the corresponding tensor factor. Now, set:
  \begin{align}
    \smallsup{\phi}[\lambda]{(i,i+1)}\coloneqq\sum_{s=1}^{\lambda_{i+1}}\smallsup{\phi}[\lambda]{(i,i+1,s)}&:\sum_{s=1}^{\lambda_{i+1}}\Lambda^{\smallsup{\lambda}{(i,i+1,s)}}E\to\Lambda^{\lambda}E,\label{Bu.mapI} \\
    \phi_{\lambda}\coloneqq\sum_{i=1}^{\ell-1}\smallsup{\phi}[\lambda]{(i,i+1)}&:\sum_{i=1}^{\ell-1}\sum_{s=1}^{\lambda_{i+1}}\Lambda^{\smallsup{\lambda}{(i,i+1,s)}}E\to\Lambda^{\lambda}E.\label{Bu.mapII}
  \end{align}
For $\lambda\in\Lambda^{+}(n)$, we have that $\coker\phi_{\lambda'}\cong L_{\lambda'}(E)$ \cite[Theorem II.2.16]{ABW}, and hence $\coker\phi_{\lambda'}\cong\nabla(\lambda)$ \cite[§2.7(5)]{D3}. We shall refer to this description as the \emph{ABW-construction} of $\nabla(\lambda)$.

Now, we review an alternative description of $\nabla(\lambda)$ due to James \cite[§26]{J}. Although James refers to this module as the \lq\lq Weyl module", it is not to be confused with the usual Weyl module $\Delta(\lambda)$ that we discussed above \cite[Theorem (4.8f)]{G}. James' construction is as follows. Recall that the symmetric algebra $S(E)$ of $E$ also has a Hopf algebra structure \cite[§I.2]{ABW}. As a slight abuse of notation, we shall once again use the symbols $\Delta$ and $\mu$ for the corresponding comultiplication and multiplication of $S(E)$ respectively. Let $\lambda$ be a partition with $\ell\coloneqq\ell(\lambda)$. For \mbox{$1\leq i<\ell$}, \mbox{$1<j\leq\ell$}, \mbox{$1\leq t\leq\lambda_{j}$}, and \mbox{$s\geq 1$}, we consider the $G$-homomorphisms $\smallsup{\Delta}[\lambda]{(j,t)}:S^{\lambda_{j}}E\to S^{t}E\otimes S^{\lambda_{j}-t}E$ and $\smallsup{\mu}[\lambda]{(i,s)}:S^{\lambda_{i}}E\otimes S^{s}E\to S^{\lambda_{i}+s}E$ coming from $\Delta$ and $\mu$ respectively. Further, for $1\leq i<j\leq\ell$, $1\leq t\leq\lambda_{j}$, we construct the \mbox{$G$-homomorphism} \mbox{$\smallsup{\psi}[\lambda]{(i,j,t)}:S^{\lambda}E\to S^{\smallsup{\lambda}{(i,j,t)}}E$} as the composition:
  \begin{equation}
    \begin{aligned}\label{General.Jam.I}
        &S^{\lambda}E\to[{1\otimes\cdots\otimes\smallsup{\Delta}[\lambda]{(j,t)}\otimes\cdots\otimes 1}]S^{\lambda_{1}}E\otimes\cdots\otimes S^{\lambda_{i}}E\otimes\cdots\otimes S^{t}E\otimes S^{\lambda_{j}-t}E\otimes\cdots\otimes S^{\lambda_{\ell}}E \to[{\bar{\sigma}}] \\
        &S^{\lambda_{1}}E\otimes\cdots\otimes S^{\lambda_{i}}E\otimes S^{t}E\otimes\cdots\otimes S^{\lambda_{j}-t}E\otimes\cdots\otimes S^{\lambda_{\ell}}E \to[{1\otimes\cdots\otimes\smallsup{\mu}[\lambda]{(i,t)}\otimes\cdots\otimes 1}]S^{\smallsup{\lambda}{(i,j,t)}}E,
    \end{aligned}
  \end{equation}
where $\bar{\sigma}$ denotes the isomorphism that permutes the corresponding tensor factors, and each $1$ refers to the identity map on the corresponding tensor factor. Now, set:
  \begin{align}
    \smallsup{\psi}[\lambda]{(i,i+1)}\coloneqq\sum_{t=1}^{\lambda_{i+1}}\smallsup{\psi}[\lambda]{(i,i+1,t)}:S^{\lambda}E\to&\sum_{t=1}^{\lambda_{i+1}}S^{\smallsup{\lambda}{(i,i+1,t)}}E,\label{Ja.map.I} \\
    \psi_{\lambda}\coloneqq\sum_{i=1}^{\ell-1}\smallsup{\psi}[\lambda]{(i,i+1)}:S^{\lambda}E\to&\sum_{i=1}^{\ell-1}\sum_{t=1}^{\lambda_{i+1}}S^{\smallsup{\lambda}{(i,i+1,t)}}E.\label{Ja.mapII}
  \end{align}
For $\lambda\in\Lambda^{+}(n)$, we have that $\nabla(\lambda)\cong\ker\psi_{\lambda}$ \cite[Theorem 26.5]{J}. We shall refer to this description as the \emph{James-construction} of $\nabla(\lambda)$.

It is important to point out that the James-construction of $\nabla(\lambda)$ may be derived from Akin, Buchsbaum, and Weyman's construction of the Weyl module $\Delta(\lambda)$ via contravariant duality \cite[§II.3]{ABW}. Similarly to \labelcref{General.Buch.I}, \labelcref{Bu.mapI}, and \labelcref{Bu.mapII}, one may define a \mbox{$G$-homomorphism} $\smallsup{\theta}[\lambda]{(i,j,t)}:D^{\smallsup{\lambda}{(i,j,t)}}E\to D^{\lambda}E$ for $1\leq i<j\leq\ell$, $1\leq t\leq\lambda_{j}$, and then construct the \mbox{$G$-homomorphism}:
  \begin{equation}\label{Buc.We.map}
    \theta_{\lambda}\coloneqq\sum_{i=1}^{\ell-1}\sum_{t=1}^{\lambda_{i+1}}\smallsup{\theta}[\lambda]{(i,i+1,t)}:\sum_{i=1}^{\ell-1}\sum_{t=1}^{\lambda_{i+1}}D^{\smallsup{\lambda}{(i,i+1,t)}}E\to D^{\lambda}E.
  \end{equation}
For $\lambda\in\Lambda^{+}(n)$, we have that $\Delta(\lambda)\cong\coker\theta_{\lambda}$ \cite[Theorem II.3.16]{ABW}. Now, recall that $\Delta(\lambda)^{\circ}\cong\nabla(\lambda)$ and that $(D^{\alpha}E)^{\circ}\cong S^{\alpha}E$ for $\alpha\in\Lambda(n)$. By taking contravariant duals, it follows that $\nabla(\lambda)\cong\ker {\theta_{\lambda}}^{\circ}$ and it is easy to check that we have the identifications ${\theta_{\lambda}}^{\circ}=\psi_{\lambda}$ and ${\smallsup{\theta}[\lambda]{(i,j,t){\ss\circ}}}=\smallsup{\psi}[\lambda]{(i,j,t)}$ for $1\leq i<j\leq\ell$, $1\leq t\leq\lambda_{j}$.

\subsection{Connections with the symmetric groups}\label{subsec.Schur}

Recall that for a partition $\lambda$ of $r$, we have the Specht module $\Sp(\lambda)$ for $\k\mathfrak{S}_{r}$. For $\lambda=(1^{r})$, we have that $\Sp(1^{r})$ is the sign representation $\sgn_{r}$ of $\k\mathfrak{S}_{r}$. We fix $n\geq r$, and we consider the {\emph{Schur functor}} \mbox{$f:M_{\k}(n,r)\rightarrow\k\mathfrak{S}_{r}\modules$}, where $fV\coloneqq V^{(1^{r})}$ for $V\in M_{\k}(n,r)$ \cite[§6.3]{G}. For $\lambda\in\Lambda^{+}(n,r)$ we have the isomorphism \mbox{$f\nabla(\lambda)\cong\Sp(\lambda)$}\mbox{\cite[(6.3c)]{G}}, and for $\alpha\in\Lambda(n,r)$ we have the \mbox{$\k\mathfrak{S}_{r}$-isomorphisms} $fS^{\alpha}E\cong M(\alpha)$ and \mbox{$f\Lambda^{\alpha}E\cong M(\alpha)\otimes\sgn_{r}\eqqcolon M_{s}(\alpha)$}, where $M_{s}(\alpha)$ denotes the {\emph{signed permutation module}} corresponding to $\alpha$ \cite[Lemma 3.4]{D1}. We set $\ell\coloneqq\ell(\lambda)$. By applying the Schur functor to the maps $\phi_{\lambda}$ and $\psi_{\lambda}$ from \labelcref{Bu.mapII} and \labelcref{Ja.mapII} respectively, we obtain the \mbox{$\k\mathfrak{S}_{r}$-homomorphisms}:
  \begin{align}
    \bar{\phi}_{\lambda}&\coloneqq f(\phi_{\lambda}):\bigoplus_{i=1}^{\ell-1}\bigoplus_{s=1}^{\lambda_{i+1}}M_{s}(\smallsup{\lambda}{(i,i+1,s)})\to M_{s}(\lambda),\label{Bu.mapIII}\\
    \bar{\psi}_{\lambda}&\coloneqq f(\psi_{\lambda}):M(\lambda)\to\bigoplus_{i=1}^{\ell-1}\bigoplus_{t=1}^{\lambda_{i+1}}M(\smallsup{\lambda}{(i,i+1,t)}).\label{Ja.mapIII}
  \end{align}
As a consequence of the exactness of the Schur functor $f$, it follows that \mbox{$\Sp(\lambda)\cong\coker\bar{\phi}_{\lambda'}$} and $\Sp(\lambda)\cong\ker\bar{\psi}_{\lambda}$. This second isomorphism is an alternative realisation of James' Kernel Intersection Theorem \cite[Corollary 17.18]{J}. These two descriptions of the Specht module $\Sp(\lambda)$ will be crucial for our considerations in this paper.

We set $S\coloneqq S_{\k}(n,r)$ for the Schur algebra. The Schur functor has a partial inverse $g:\k\mathfrak{S}_{r}\modules\rightarrow S\modules$ with $gW\coloneqq Se\otimes_{eSe}W$ for $W\in\k\mathfrak{S}_{r}\modules$, where $e$ denotes a certain idempotent of $S$ \cite[(6.2c)]{G}. This functor is a right-inverse of $f$ and it is right-exact. Given $V\in M_{\k}(n,r)$ and $W\in\k\mathfrak{S}_{r}\modules$, there is a \mbox{$\k$-linear} isomorphism of the form $\Hom_{G}(gW,V)\cong\Hom_{\k\mathfrak{S}_{r}}(W,fV)$. For $\alpha\in\Lambda(n,r)$ one has that $gM(\alpha)\cong S^{\alpha}E$ \cite[Appendix A]{DG2}, and for $\lambda\in\Lambda^{+}(n,r)$ and $p\neq 2$ one has that $g\Sp(\lambda)\cong\nabla(\lambda)$ \cite[Proposition 10.6(i)]{D1}, \cite[Theorem 1.1]{McD}. Further results related to the properties of $g$ will be proved in \Cref{sec.end.alg}, including a new short proof of the fact that $g\Sp(\lambda)\cong\nabla(\lambda)$ for $p\neq 2$.

\section{Endomorphism algebras}\label{sec.end.alg}

\subsection{General Results}\label{subsec.gen.res}

From now on we fix $n\geq r$. Note that for $\lambda\in\Lambda^{+}(n,r)$ we have that $\lambda'\in\Lambda^{+}(n,r)$. First we point out some new properties of the functor $g$ and then we utilise the two different descriptions of the Specht module $\Sp(\lambda)$ to introduce a new description of the endomorphism algebra of $\Sp(\lambda)$.

\begin{Proposition}\label{inverse.g}
  Assume that $p\neq 2$. Then:
    \begin{enumerate}[label=(\roman*), font=\normalfont, ref=(\roman*)]
      \item\label{inverse.g.item.1}For $\alpha\in\Lambda(n,r)$, we have $gM_{s}(\alpha)\cong\Lambda^{\alpha}E$.
      \item\label{inverse.g.item.2}For $\lambda\in\Lambda^{+}(n,r)$, we have $g\Sp(\lambda)\cong\nabla(\lambda)$.
    \end{enumerate}
\end{Proposition}
\begin{proof}
  \labelcref{inverse.g.item.1} Recall that for $\beta\in\Lambda(n,r)$ and $V\in M_{\k}(n,r)$, we have a \mbox{$\k$-isomorphism} $\Hom_{G}(V,S^{\beta}E)\cong V^{\beta}$, and so in particular $\dim V^{\beta}=\dim\Hom_{G}(V,S^{\beta}E)$. Moreover, $fS^{\alpha}E\cong M(\alpha)$ and so it follows that:
      \[
        \Hom_{G}(gM_{s}(\alpha),S^{\beta}E)\cong\Hom_{\k\mathfrak{S}_{r}}(M_{s}(\alpha),fS^{\beta}E)\cong\Hom_{\k\mathfrak{S}_{r}}(M_{s}(\alpha),M(\beta)).
      \]
  Now, since $p\neq 2$, the dimension of $\Hom_{\k\mathfrak{S}_{r}}(M_{s}(\alpha),M(\beta))$ does not depend on the value of $p$ \cite[Theorem 3.3(ii)]{DJ}, and so in order to calculate the dimension of $gM_{s}(\alpha)^{\beta}$, we may assume that $p=0$. However, in characteristic $0$, the functors $f$ and $g$ are inverse equivalences of categories and so $gM_{s}(\alpha)\cong\Lambda^{\alpha}E$. Therefore, for $p\neq 2$, we deduce that $\dim gM_{s}(\alpha)^{\beta}=\dim\Lambda^{\alpha}E^{\beta}$ for all $\beta\in\Lambda(n,r)$. Now, recall that for $V\in M_{\k}(n,r)$, we have the weight space decomposition $V=\bigoplus_{\beta\in\Lambda(n,r)}V^{\beta}$ \cite[(3.2c)]{G}, and so it follows that, for $p\neq 2$, we have $\dim gM_{s}(\alpha)=\dim\Lambda^{\alpha}E$.

  Now, we have that $M(1^{r})\cong eSe$ and so $gM(1^{r})\cong Se\otimes_{eSe}eSe\cong Se\cong E^{\otimes r}$ \cite[(6.4f)]{G}. For $\alpha\in\Lambda(n,r)$ we have a surjective \mbox{$G$-homomorphism} $E^{\otimes r}\to\Lambda^{\alpha}E$ and so via the Schur functor, we get a surjective \mbox{$\k\mathfrak{S}_{r}$-homomorphism} $M(1^{r})\to M_{s}(\alpha)$. The functor $g$, being right-exact, preserves surjections, and so the \mbox{$G$-homomorphism} $gM(1^{r})\to gM_{s}(\alpha)$ is surjective. We consider the commutative diagram:
    \[
      \begin{tikzcd}
        gM(1^{r})\arrow{r}{\cong} \arrow{d}[swap]{} & E^{\otimes r} \arrow{d}{} \\
        gM_{s}(\alpha)\arrow[swap]{r}{} & \Lambda^{\alpha}E
      \end{tikzcd},
    \]
  where the horizontal maps are induced from the $\k\mathfrak{S}_{r}$-inclusions $M(1^{r})\cong fE^{\otimes r}\to E^{\otimes r}$ and $M_{s}(\alpha)\cong f\Lambda^{\alpha}E\to\Lambda^{\alpha}E$. The top horizontal map is an isomorphism and the \mbox{right-hand} vertical map is surjective, and so the bottom horizontal map is hence surjective. Since $\dim gM_{s}(\alpha)=\dim\Lambda^{\alpha}E$ away from characteristic $2$, we obtain \mbox{$gM_{s}(\alpha)\cong\Lambda^{\alpha}E$} for $p\neq 2$.

  \labelcref{inverse.g.item.2} Recall that $\nabla(\lambda)\cong\coker\phi_{\lambda'}$, where $\phi_{\lambda'}:K(\lambda')\to\Lambda^{\lambda'}E$ and $K(\lambda')$ is the direct sum of tensor products of exterior powers given in \labelcref{Bu.mapII}, where here we replace the partition $\lambda$ with $\lambda'$. By applying the Schur functor $f$ to $\phi_{\lambda'}$, we obtain the \mbox{$\k\mathfrak{S}_{r}$-homomorphism} $\bar{\phi}_{\lambda'}:\bar{K}(\lambda')\to M_{s}(\lambda')$, where $\bar{K}(\lambda')$ is the direct sum of signed permutation modules given in \labelcref{Ja.mapII}, again substituting $\lambda$ with $\lambda'$. Also, recall that $\Sp(\lambda)\cong\coker\bar{\phi}_{\lambda'}$. By \partCref{inverse.g.item.1}, we have that $gM_{s}(\lambda')\cong\Lambda^{\lambda'}E$ and so $g\bar{K}(\lambda')\cong K(\lambda')$. Hence, we obtain the commutative diagram:
    \[
      \begin{tikzcd}
        g\bar{K}(\lambda')\arrow{r}{g(\bar{\phi}_{\lambda'})} \arrow{d}[swap]{\cong} & gM_{s}(\lambda')\arrow{d}{\cong} \\
        K(\lambda')\arrow[]{r}{\phi_{\lambda'}} & \Lambda^{\lambda'}E
      \end{tikzcd}.
    \]
  The image of $g(\bar{\phi}_{\lambda'})$ is mapped isomorphically onto the image of $\phi_{\lambda'}$, and so in particular $\coker\phi_{\lambda'}\cong\coker g(\bar{\phi}_{\lambda'})$. Finally, $g$ preserves cokernels since it is right-exact, and so we deduce that $\nabla(\lambda)\cong\coker\phi_{\lambda'}\cong\coker g(\bar{\phi}_{\lambda'})\cong g \, \coker\bar{\phi}_{\lambda'}\cong g \Sp(\lambda)$.
\end{proof}

\begin{Lemma}\label{end.alg}
  Let $\alpha,\beta\in\Lambda(n,r)$. Then:
    \begin{enumerate}[label=(\roman*), font=\normalfont, ref=(\roman*)]
      \item\label{end.alg.item.1}$\Hom_{\k\mathfrak{S}_{r}}(M(\alpha),M(\beta))\cong\Hom_{G}(S^{\alpha}E,S^{\beta}E)\cong(S^{\alpha}E)^{\beta}$.
      \item\label{end.alg.item.2}For $p\neq 2$, we have $\Hom_{\k\mathfrak{S}_{r}}(M_{s}(\alpha),M(\beta))\cong\Hom_{G}(\Lambda^{\alpha}E,S^{\beta}E)\cong(\Lambda^{\alpha}E)^{\beta}$.
    \end{enumerate}
\end{Lemma}
\begin{proof}
  Recall that for $V\in M_{\k}(n,r)$ and $W\in\k\mathfrak{S}_{r}\modules$, we have a \mbox{$\k$-isomorphism} of the form $\Hom_{G}(gW,V)\cong\Hom_{\k\mathfrak{S}_{r}}(W,fV)$. \hyperref[{end.alg}]{Parts (i)-(ii)} then both follow from our comments in \secCref{subsec.pol.rep}, \secCref{subsec.Schur}, and \myCref{inverse.g}[inverse.g.item.1].
\end{proof}

\begin{Lemma}\label{end.alg.Sp}
  Let $\lambda\in\Lambda^{+}(n,r)$. Then:
    \begin{enumerate}[label=(\roman*), font=\normalfont, ref=(\roman*)]
      \item\label{end.alg.Sp.item.1}There is a \mbox{$\k$-isomorphism}:
        \[ \End_{\k\mathfrak{S}_{r}}(\Sp(\lambda))\cong\{h\in\Hom_{\k\mathfrak{S}_{r}}(M_{s}(\lambda'),M(\lambda))\mid h\circ\bar{\phi}_{\lambda'}=0\settext{and}\bar{\psi}_{\lambda}\circ h=0\}. \]
      \item\label{end.alg.Sp.item.2}In particular, when $p=2$, there is a \mbox{$\k$-isomorphism}:
        \[ \End_{\k\mathfrak{S}_{r}}(\Sp(\lambda))\cong\{h\in\Hom_{\k\mathfrak{S}_{r}}(M(\lambda'),M(\lambda))\mid h\circ\bar{\phi}_{\lambda'}=0\settext{and}\bar{\psi}_{\lambda}\circ h=0\}. \]
    \end{enumerate}
\end{Lemma}
\begin{proof}
  \PartCref{end.alg.Sp.item.1} follows immediately from the two descriptions of the Specht module: $\Sp(\lambda)\cong\coker\bar{\phi}_{\lambda'}$ and $\Sp(\lambda)\cong\ker\bar{\psi}_{\lambda}$ from \secCref{subsec.Schur}. \PartCref{end.alg.Sp.item.2} then follows from \partCref{end.alg.Sp.item.1} and the fact that the permutation module and the signed permutation module coincide in characteristic $2$.
\end{proof}

Recall the \mbox{$G$-homomorphisms} $\smallsup{\phi}[\lambda]{(i,j,s)}$ and $\smallsup{\psi}[\lambda]{(i,j,t)}$ from \labelcref{General.Buch.I} and \labelcref{General.Jam.I} respectively.

\begin{Lemma}\label{map.images}
  Let $\lambda\in\Lambda^{+}(n)$ with $\ell\coloneqq\ell(\lambda)$. Then:
  \begin{enumerate}[label=(\roman*), font=\normalfont, ref=(\roman*)]
    \item\label{map.images.item.1}$\im\smallsup{\phi}[\lambda]{(i,j,s)}\subseteq\im\phi_{\lambda}$ for $1\leq i<j\leq\ell$, $1\leq s\leq\lambda_{j}$.
    \item\label{map.images.item.2}$\ker\psi_{\lambda}\subseteq\ker\smallsup{\psi}[\lambda]{(i,j,t)}$ for $1\leq i<j\leq\ell$, $1\leq t\leq\lambda_{j}$.
  \end{enumerate}
\end{Lemma}
\begin{proof}
  \PartCref{map.images.item.1} follows from \cite[Lemma II.2.3, Theorem II.2.16]{ABW}, where we note that one may replace $i+1$ with any $j>i$ in the statement and proof of \cite[Lemma II.2.3]{ABW} without any harm. For \partCref{map.images.item.2}, we use the ABW-construction of the Weyl module \mbox{$\Delta(\lambda)$ \labelcref{Buc.We.map}}. By \cite[Theorem II.3.16]{ABW} and the comment before\mbox{\cite[Definition II.3.4]{ABW}}, we deduce that $\im\smallsup{\theta}[\lambda]{(i,j,t)}\subseteq\im\theta_{\lambda}$ for $1\leq i<j\leq\ell$ and $1\leq t\leq\lambda_{j}$. Taking contravariant duals, we have that $\ker{\theta_{\lambda}}^{\circ}\subseteq\ker\smallsup{\theta}[\lambda]{(i,j,t){\ss\circ}}$ for all such $i,j,t$. The result follows by recalling the identifications ${\theta_{\lambda}}^{\circ}=\psi_{\lambda}$ and $\smallsup{\theta}[\lambda]{(i,j,t){\ss\circ}}=\smallsup{\psi}[\lambda]{(i,j,t)}$ from \secCref{subsec.pol.rep}.
\end{proof}

Let $\lambda\in\Lambda^{+}(n,r)$. By applying the Schur functor $f$ to the maps $\smallsup{\phi}[\lambda]{(i,j,s)}$ and $\smallsup{\psi}[\lambda]{(i,j,t)}$ of \labelcref{General.Buch.I} and \labelcref{General.Jam.I} respectively, we obtain the \mbox{$\k\mathfrak{S}_{r}$-homomorphisms}:
  \[
    \smallsup{\bar{\phi}}[\lambda]{(i,j,s)}:M_{s}(\smallsup{\lambda}{(i,j,s)})\to M_{s}(\lambda),\qquad\smallsup{\bar{\psi}}[\lambda]{(i,j,t)}:M(\lambda)\to M(\smallsup{\lambda}{(i,j,t)}).
  \]

\begin{Remark}\label{rem.boundaries}
  We may view any partition $\lambda\in\Lambda^{+}(n,r)$ as an $n$-tuple by appending an appropriate number of zeros to $\lambda$. Accordingly, we may relax the dependence on $\ell(\lambda)$ of the maps $\bar{\phi}_{\lambda}$ and $\bar{\psi}_{\lambda}$. We do so by setting $\smallsup{\bar{\phi}}[\lambda]{(i,j,s)}\coloneqq 0$ and $\smallsup{\bar{\psi}}[\lambda]{(i,j,t)}\coloneqq 0$ if $\ell(\lambda)<j\leq n$.
\end{Remark}

By \myCref{end.alg.Sp}[end.alg.Sp.item.2] and \Cref{map.images}, we obtain the following \nameCref{gen.rel}:

\begin{Corollary}\label{gen.rel}
  Assume that $\Char\k=2$ and let $\lambda\in\Lambda^{+}(n,r)$. Then the endomorphism algebra of $\Sp(\lambda)$ may be identified with the \mbox{$\k$-subspace} of $\Hom_{\k\mathfrak{S}_{r}}(M(\lambda'),M(\lambda))$ consisting of those elements $h$ that satisfy:
    \begin{enumerate}[label=(\roman*), font=\normalfont, ref=(\roman*)]
      \item\label{gen.rel.item.1}$h\circ\smallsup{\bar{\phi}}[\lambda']{(i,j,s)}=0$ for $1\leq i<j\leq n$ and $1\leq s\leq\lambda'_{j}$,
      \item\label{gen.rel.item.2}$\smallsup{\bar{\psi}}[\lambda]{(i,j,t)}\circ h=0$ for $1\leq i<j\leq n$ and $1\leq t\leq\lambda_{j}$.
    \end{enumerate}
\end{Corollary}

\subsection{A concrete description}\label{subsec.matrices}

From now on we shall assume that the underlying field $\k$ has characteristic $2$. We write $[r]\coloneqq\{1,\ldots,r\}$ and as always we assume that $n\geq r$. First, we provide a matrix description of a \mbox{$\k$-basis} of $\Hom_{\k\mathfrak{S}_{r}}(M(\alpha),M(\beta))$ for $\alpha,\beta\in\Lambda(n,r)$, and then we shall utilise this description to obtain some crucial information regarding the endomorphism algebra of $\Sp(\lambda)$.

We write $M_{n\times n}(\mathbb{N})$ for the set of $(n\times n)$-matrices with non-negative integer entries. Let $\{e_{i}\mid 1\leq i\leq n\}$ be the standard basis of column vectors of $E$. Then, for $\alpha\in\Lambda(n,r)$, we consider the \mbox{$\k$-basis} $\{e_{1}^{a_{11}}e_{2}^{a_{12}}\ldots e_{n}^{a_{1n}}\otimes\cdots\otimes e_{1}^{a_{n1}}e_{2}^{a_{n2}}\ldots e_{n}^{a_{nn}}\mid\sum_{j}a_{ij}=\alpha_{i}\}$ of $S^{\alpha}E$, where the $i$th tensor factor is defined to be $1$ if $\alpha_{i}=0$ for some $1\leq i\leq n$. We may parametrise this \mbox{$\k$-basis} by the set of all elements of $M_{n\times n}(\mathbb{N})$ whose sequence of row-sums is equal to $\alpha$. Accordingly, for $\beta\in\Lambda(n,r)$, the $\beta$-weight space $(S^{\alpha}E)^{\beta}$ has a \mbox{$\k$-basis} parametrised by the set of all matrices in $M_{n\times n}(\mathbb{N})$ whose sequence of row-sums is equal to $\alpha$, and whose sequence of column-sums is equal to $\beta$. On the other hand, the permutation module $M(\alpha)$ has a \mbox{$\k$-basis} consisting of all ordered sequences of the form $(\bmath{x}_{1}{\mid}\ldots{\mid}\bmath{x}_{n})$, where each $\bmath{x}_{i}=(x_{i1},x_{i2},\ldots,x_{i\alpha_{i}})$ is an unordered sequence with terms from $[r]$, that satisfy the property that for each $k\in[r]$, there is a unique pair $(i,j)$ with $x_{ij}=k$. Here $\bmath{x}_{i}$ denotes the zero sequence whenever $\alpha_{i}=0$.

We set $\Tab(\alpha,\beta)\coloneqq\{A=(a_{ij})_{i,j}\in M_{n\times n}(\mathbb{N})\mid\sum_{j}a_{ij}=\alpha_{i},\sum_{i}a_{ij}=\beta_{j}\}$. We associate to each $A\in\Tab(\alpha,\beta)$, a homomorphism $\rho[A]\in\Hom_{\k\mathfrak{S}_{r}}(M(\alpha),M(\beta))$. We do so as follows: Given a basis element $\bmath{x}\coloneqq(\bmath{x}_{1}{\mid}\ldots{\mid}\bmath{x}_{n})\in M(\alpha)$, we set $\rho[A](\bmath{x})$ to be the sum of all basis elements of $M(\beta)$ that are obtained from $\bmath{x}$ by moving, in concert, $a_{ij}$ entries from its \mbox{$i$th-position} $\bmath{x}_{i}$ to its \mbox{$j$th-position} $\bmath{x}_{j}$ in all possible ways, and for every $1\leq i,j\leq n$. The set $\{\rho[A]\mid A\in\Tab(\alpha,\beta)\}$ is linearly independent. Indeed, take any linear combination of the $\rho[A]$s, say $h=\sum_{A}h[A]\rho[A]$ ($h[A]\in\k$), along with any basis element $\bmath{x}$ of $M(\alpha)$, and then consider the coefficients of the basis elements of $M(\beta)$ in $h(\bmath{x})$. The linear independence of the $\rho[A]$s along with \myCref{end.alg}[end.alg.item.1] give that the set $\{\rho[A]\mid A\in\Tab(\alpha,\beta)\}$ forms a \mbox{$\k$-basis} of $\Hom_{\k\mathfrak{S}_{r}}(M(\alpha),M(\beta))$. Accordingly, for $h\in\Hom_{\k\mathfrak{S}_{r}}(M(\alpha),M(\beta))$ and $A\in\Tab(\alpha,\beta)$, we shall denote by $h[A]\in\k$ the coefficient of $\rho[A]$ in $h$ so that $h=\sum_{A\in\Tab(\alpha,\beta)}h[A]\rho[A]$.

\begin{Examples}\label{exam.basis.elem}
  Let $\lambda\in\Lambda^{+}(n,r)$. For $1\leq i,j\leq n$, denote by $E_{ij}\in M_{n\times n}(\mathbb{N})$ the matrix with a $1$ in its $(i,j)$th-position and $0$s elsewhere. Notice that:
    \begin{enumerate}[label=(\roman*), font=\normalfont, ref=(\roman*)]
      \item\label{exam.basis.elem.item.1} $\smallsup{\bar{\phi}}[\lambda]{(i,j,s)}=\rho[A]$, where $A\coloneqq\diag(\lambda_{1},\ldots,\lambda_{i},\ldots,\lambda_{j}-s,\ldots,\lambda_{n})+s E_{ij}$.
      \item\label{exam.basis.elem.item.2} $\smallsup{\bar{\psi}}[\lambda]{(i,j,t)}=\rho[B]$, where $B\coloneqq\diag(\lambda_{1},\ldots,\lambda_{i},\ldots,\lambda_{j}-t,\ldots,\lambda_{n})+t E_{ji}$.
    \end{enumerate}
\end{Examples}

\begin{Remark}\label{rem.trans}
  Recall the \mbox{$\k$-basis} $\{\rho[A]\mid A\in\Tab(\alpha,\beta)\}$ of $\Hom_{\k\mathfrak{S}_{r}}(M(\alpha),M(\beta))$. For $A\in M_{n\times n}(\mathbb{N})$, we write $A'\in M_{n\times n}(\mathbb{N})$ for the transpose matrix of $A$. If $A\in\Tab(\alpha,\beta)$, then it is clear that $ A'\in\Tab(\beta,\alpha)$. Moreover, the set $\{\rho[A']\mid A\in\Tab(\alpha,\beta)\}$ forms a \mbox{$\k$-basis} of $\Hom_{\k\mathfrak{S}_{r}}(M(\beta),M(\alpha))$.
\end{Remark}

Now, for $\alpha\in\Lambda(n,r)$, recall that the permutation module $M(\alpha)$ is self-dual. We write $d_{\alpha}:M(\alpha)\to M(\alpha)^{*}$ for the \mbox{$\k\mathfrak{S}_{r}$-isomorphism} that sends each basis element $\bmath{x}$ of $M(\alpha)$ to the corresponding basis element of $M(\alpha)^{*}$ dual to $\bmath{x}$. We shall denote by $\zeta_{\alpha,\beta}:\Hom_{\k\mathfrak{S}_{r}}(M(\alpha),M(\beta))\to\Hom_{\k\mathfrak{S}_{r}}(M(\beta)^{*},M(\alpha)^{*})$ the natural \mbox{$\k$-isomorphism}, and by $\eta_{\alpha,\beta}:\Hom_{\k\mathfrak{S}_{r}}(M(\alpha),M(\beta))\to\Hom_{\k\mathfrak{S}_{r}}(M(\beta),M(\alpha))$ the \mbox{$\k$-isomorphism} with $\eta_{\alpha,\beta}(h)=d_{\alpha}^{-1}\circ\zeta_{\alpha,\beta}(h)\circ d_{\beta}$ for $h\in\Hom_{\k\mathfrak{S}_{r}}(M(\alpha),M(\beta))$.

\begin{Lemma}\label{lem.trans}
  Let $\alpha,\beta\in\Lambda(n,r)$. Then $\eta_{\alpha,\beta}(\rho[A])=\rho[A']$ for all $A\in\Tab(\alpha,\beta)$.
\end{Lemma}
\begin{proof}
  This is a simple calculation which we leave to the reader. \qedhere
\end{proof}

\begin{Definition}\label{def.trans.not}
  For $h\in\Hom_{\k\mathfrak{S}_{r}}(M(\alpha),M(\beta))$, we shall denote by $h'$ the homomorphism $\eta_{\alpha,\beta}(h)\in\Hom_{\k\mathfrak{S}_{r}}(M(\beta),M(\alpha))$ and call it the \emph{transpose homomorphism of $h$}.
\end{Definition}

Notice that if $h=\sum_{A\in\Tab(\alpha,\beta)}h[A]\rho[A]$, then $h'=\sum_{A\in\Tab(\alpha,\beta)}h[A]\rho[A']$ by \Cref{lem.trans}.

\begin{Lemma}\label{lem.trans.comp}
  Let $\alpha,\beta,\gamma\in\Lambda(n,r)$. Then we have the identity $(h_{2}\circ h_{1})'=h'_{1}\circ h'_{2}$ for all \mbox{$h_{1}\in\Hom_{\k\mathfrak{S}_{r}}(M(\alpha),M(\beta))$} and \mbox{$h_{2}\in\Hom_{\k\mathfrak{S}_{r}}(M(\beta),M(\gamma))$}.
\end{Lemma}
\begin{proof}
  Since $\zeta_{\alpha,\gamma}(h_{2}\circ h_{1})=\zeta_{\alpha,\beta}(h_{1})\circ\zeta_{\beta,\gamma}(h_{2})$, we have:
    \begin{align*}
      (h_{2}\circ h_{1})'&=d_{\alpha}^{-1}\circ\zeta_{\alpha,\beta}(h_{1})\circ\zeta_{\beta,\gamma}(h_{2})\circ d_{\gamma} \\
      &=(d_{\alpha}^{-1}\circ\zeta_{\alpha,\beta}(h_{1})\circ d_{\beta})\circ(d_{\beta}^{-1}\circ\zeta_{\beta,\gamma}(h_{2})\circ d_{\gamma})=h'_{1}\circ h'_{2}. \qedhere
    \end{align*}
\end{proof}

\begin{Lemma}\label{lem.trans.specht}
  Let $\lambda\in\Lambda^{+}(n,r)$ and $h\in\Hom_{\k\mathfrak{S}_{r}}(M(\lambda'),M(\lambda))$. Then:
    \begin{enumerate}[label=(\roman*), font=\normalfont, ref=(\roman*)]
      \item\label{lem.trans.specht.item.1}$(h\circ\smallsup{\bar{\phi}}[\lambda']{(i,j,s)})'=\smallsup{\bar{\psi}}[\lambda']{(i,j,s)}\circ h'$.
      \item\label{lem.trans.specht.item.2}$(\smallsup{\bar{\psi}}[\lambda]{(i,j,t)}\circ h)'=h'\circ\smallsup{\bar{\phi}}[\lambda]{(i,j,t)}$.
      \item\label{lem.trans.specht.item.3}The map $\eta_{\lambda',\lambda}$ induces a \mbox{$\k$-isomorphism} $\bar{\eta}_{\lambda}:\End_{\k\mathfrak{S}_{r}}(\Sp(\lambda))\to\End_{\k\mathfrak{S}_{r}}(\Sp(\lambda'))$.
    \end{enumerate}
\end{Lemma}
\begin{proof}
  By \Cref{lem.trans} and the examples in \Cref{exam.basis.elem}, it follows that $(\smallsup{\bar{\phi}}[\lambda]{(i,j,t)})'=\smallsup{\bar{\psi}}[\lambda]{(i,j,t)}$. Now, \hyperref[{lem.trans.specht}]{parts (i)-(ii)} follow directly from \Cref{lem.trans.comp}. For \partCref{lem.trans.specht.item.3}, notice that \Cref{end.alg.Sp} gives that any element $\bar{h}\in\End_{\k\mathfrak{S}_{r}}(\Sp(\lambda))$ may be identified with a homomorphism \mbox{$h\in\Hom_{\k\mathfrak{S}_{r}}(M(\lambda'),M(\lambda))$} such that $h\circ\smallsup{\bar{\phi}}[\lambda']{(i,i+1,s)}=0$ for $1\leq i<n$, $1\leq s\leq\lambda'_{i+1}$, and also $\smallsup{\bar{\psi}}[\lambda]{(i,i+1,t)}\circ h=0$ for $1\leq i<n$, $1\leq t\leq\lambda_{i+1}$. By \hyperref[{lem.trans.specht}]{parts (i)-(ii)}, we deduce that $\smallsup{\bar{\psi}}[\lambda']{(i,i+1,s)}\circ h'=0$ and $h'\circ\smallsup{\bar{\phi}}[\lambda]{(i,i+1,t)}=0$ for all such $i,s,t$ and so $h'$ induces an endomorphism of $\Sp(\lambda')$, $\bar{h'}$ say. Therefore, it follows that the map $\eta_{\lambda',\lambda}$ induces a \mbox{$\k$-homomorphism} \mbox{$\bar{\eta}_{\lambda}: \End_{\k\mathfrak{S}_{r}}(\Sp(\lambda))\to\End_{\k\mathfrak{S}_{r}}(\Sp(\lambda'))$} with $\bar{h}\mapsto\bar{h'}$. By applying the same procedure to the map $\eta_{\lambda,\lambda'}$, we see that $\bar{\eta}_{\lambda}$ is a \mbox{$\k$-isomorphism} with inverse $\bar{\eta}_{\lambda'}$ as required.
\end{proof}

For $A=(a_{ij})_{i,j}\in M_{n\times n}(\mathbb{Z})$ and $1\leq k,l\leq n$, we shall write $\smallsup{A}{(k,l)}$ for the element of $M_{n\times n}(\mathbb{Z})$ with entries given by $\smallsup{a}[ij]{(k,l)}\coloneqq a_{ij}+\smallsub{\delta}{(i,j),(k,l)}$, and $\smallsub{A}{(k,l)}$ for the element of $M_{n\times n}(\mathbb{Z})$ with entries given by $\smallsub{a}{(k,l)}_{ij}\coloneqq a_{ij}-\smallsub{\delta}{(i,j),(k,l)}$. Let $\alpha,\beta\in\Lambda(n,r)$ with $A\in\Tab(\alpha,\beta)$, and let $1\leq i<j\leq n$, $1\leq k,l\leq n$. Note that \mbox{$\smallsubsup{A}{(j,l)}{(i,l)}\in\Tab(\smallsup{\alpha}{(i,j,1)},\beta)$} if $a_{jl}\neq 0$, whilst \mbox{$\smallsubsup{A}{(k,j)}{(k,i)}\in\Tab(\alpha,\smallsup{\beta}{(i,j,1)})$} if $a_{kj}\neq 0$.

\bigskip

Henceforth, we denote by $\mathcal{T}_{\lambda}$ the set $\Tab(\lambda',\lambda)$ for $\lambda\in\Lambda^{+}(n,r)$.

\begin{Lemma}\label{lem.compos}
  Let $\lambda\in\Lambda^{+}(n,r)$ and $1\leq i<j\leq n$. For $A\in\mathcal{T}_{\lambda}$ we have:
    \begin{enumerate}[label=(\roman*), font=\normalfont, ref=(\roman*)]
      \item\label{lem.compos.item.1}$\rho[A]\circ\smallsup{\bar{\phi}}[\lambda']{(i,j,1)}=\sum_{l}(a_{il}+1)\rho\sq{\smallsubsup{A}{(j,l)}{(i,l)}}$,
      where the sum is over all $l$ such that $a_{jl}\neq 0$.
      \item\label{lem.compos.item.2}$\smallsup{\bar{\psi}}[\lambda]{(i,j,1)}\circ\rho[A]=\sum_{k}(a_{ki}+1)\rho\sq{\smallsubsup{A}{(k,j)}{(k,i)}}$,
      where the sum is over all $k$ such that $a_{kj}\neq 0$.
    \end{enumerate}
\end{Lemma}
\begin{proof}
  We shall only prove \partCref{lem.compos.item.1} since \partCref{lem.compos.item.2} is similar. We may assume that $j\leq\ell(\lambda')$. Fix $1\leq i<j\leq\ell(\lambda')$, and we denote by $\bmath{x}\coloneqq(\bmath{x}_{1}{\mid}\ldots{\mid}\bmath{x}_{i}{\mid}\ldots{\mid}\bmath{x}_{j}{\mid}\ldots{\mid}\bmath{x}_{n})$ a basis element of $M(\smallsup{\lambda'}{(i,j,1)})$, where $\bmath{x}_{i}=(x_{i1},\ldots,x_{i(\lambda'_{i}+1)})$ say. Then $\smallsup{\bar{\phi}}[\lambda']{(i,j,1)}(\bmath{x})=\sum_{\smash[t]{k=1}}^{\smash[t]{\lambda'_{i}}+1}\bmath{x}^{k}$, where $\bmath{x}^{k}$ denotes the basis element of $M(\lambda')$ that is obtained from $\bmath{x}$ by omitting the entry $x_{ik}$ from the sequence $\bmath{x}_{i}$ and placing it in the (unordered) sequence $\bmath{x}_{j}$. For \mbox{$1\leq k\leq\lambda'_{i}+1$}, we have $\rho[A](\bmath{x}^{k})=\sum_{t}c_{kt}\bmath{z}[t]$, where the $\bmath{z}[t]$ are the basis elements of $M(\lambda)$ and the $c_{kt}$ are constants with $c_{kt}\in\{0,1\}$. Then $\rho[A]\circ\smallsup{\bar{\phi}}[\lambda']{(i,j,1)}(\bmath{x})=\sum_{t}c_{t}\bmath{z}[t]$ where $c_{t}\coloneqq\sum_{\smash[t]{k=1}}^{\smash[t]{\lambda'_{i}+1}}c_{kt}$. Now, fix $1\leq k\leq\lambda'_{i}+1$ and some $s$ with $c_{ks}=1$. Then, suppose that the entry $x_{ik}$ appears in the $l$th-position $\bmath{z}[s]_{l}$ of $\bmath{z}[s]$ and hence $a_{jl}\neq 0$. Note that the sequence $\bmath{z}[s]_{l}$ contains $a_{il}$ entries from $\{x_{i1},\ldots,x_{i(k-1)},x_{i(k+1)},\ldots,x_{i(\lambda'_{i}+1)}\}$. If $x_{iv}$ is such an entry with $v\neq k$, then $c_{vs}=1$. On the other hand, given $1\leq q\leq\lambda'_{i}+1$, if $x_{iq}$ does not appear as an entry in $\bmath{z}[s]_{l}$, then $c_{qs}=0$. It follows that $c_{s}=a_{il}+1$. Meanwhile, given $1\leq l'\leq n$, $\bmath{z}[s]$ appears in $\rho\sq{{\smallsubsup{A}{(j,\smash[t]{l'})}{(i,\smash[t]{l'})}}}(\bmath{x})$ if and only if $l'=l$, in which case it appears with a coefficient of $1$. The result follows.
\end{proof}

\begin{Lemma}\label{lem.gen.rel}
  Let $\lambda\in\Lambda^{+}(n,r)$ and consider a homomorphism $h\in\Hom_{\k\mathfrak{S}_{r}}(M(\lambda'),M(\lambda))$ with $h=\sum_{A\in\mathcal{T}_{\lambda}}h[A]\rho[A]$. Then for $1\leq i<j\leq n$, we have:
    \begin{enumerate}[label=(\roman*), font=\normalfont, ref=(\roman*)]
      \item\label{lem.gen.rel.item.1}$h\circ\smallsup{\bar{\phi}}[\lambda']{(i,j,1)}=0$ if and only if $\sum_{l}b_{il}h\sq{\smallsubsup{B}{(i,l)}{(j,l)}}=0$ for all $B\in\Tab(\smallsup{\lambda'}{(i,j,1)},\lambda)$.
      \item\label{lem.gen.rel.item.2}$\smallsup{\bar{\psi}}[\lambda]{(i,j,1)}\circ h=0$ if and only if $\sum_{k}d_{ki}h\sq{\smallsubsup{D}{(k,i)}{(k,j)}}=0$ for all $D\in\Tab(\lambda',\smallsup{\lambda}{(i,j,1)})$.
    \end{enumerate}
\end{Lemma}
\begin{proof}
  We shall only prove \partCref{lem.gen.rel.item.1} since \partCref{lem.gen.rel.item.2} is similar. By \Cref{lem.compos} we have:
    \begin{align*}
      h\circ\smallsup{\bar{\phi}}[\lambda']{(i,j,1)}=&\sum_{A\in\mathcal{T}_{\lambda}}h[A](\rho[A]\circ\smallsup{\bar{\phi}}[\lambda']{(i,j,1)})=\sum_{A\in\mathcal{T}_{\lambda}}h[A]\left(\sum_{l}(a_{il}+1)\rho\sq{\smallsubsup{A}{(j,l)}{(i,l)}}\right)\\
      =&\sum_{A\in\mathcal{T}_{\lambda}}\sum_{l}(a_{il}+1)h[A]\rho\sq{\smallsubsup{A}{(j,l)}{(i,l)}}=\sum_{B\in\Tab(\smallsup{\lambda'}{(i,j,1)},\lambda)}\left(\sum_{l}b_{il}h\sq{\smallsubsup{B}{(i,l)}{(j,l)}}\right)\rho[B].
    \end{align*}
  The result now follows from the linear independence of $\{\rho[B]\mid B\in\Tab(\smallsup{\lambda'}{(i,j,1)},\lambda)\}$.
\end{proof}

\begin{Definition}\label{def.first.rel}
  Let $\lambda\in\Lambda^{+}(n,r)$. We say that an element $h\in\Hom_{\k\mathfrak{S}_{r}}(M(\lambda'),M(\lambda))$ is \emph{relevant} if $h\circ\smallsup{\bar{\phi}}[\lambda']{(i,j,1)}=0$ and $\smallsup{\bar{\psi}}[\lambda]{(i,j,1)}\circ h=0$ for all $1\leq i<j\leq n$.
\end{Definition}

Denote by $\Rel_{\k\mathfrak{S}_{r}}(M(\lambda'),M(\lambda))$ the $\k$-subspace of \mbox{$\Hom_{\k\mathfrak{S}_{r}}(M(\lambda'),M(\lambda))$} consisting of the relevant homomorphisms $M(\lambda')\to M(\lambda)$. The following \namecref{rem.emb.first.rel} is clear:

\begin{Remark}\label{rem.emb.first.rel}
  Let $\lambda\in\Lambda^{+}(n,r)$. Note that there is a $\k$-embedding of the endomorphism algebra of $\Sp(\lambda)$ into the $\k$-space $\Rel_{\k\mathfrak{S}_{r}}(M(\lambda'),M(\lambda))$.
\end{Remark}

Now, by \Cref{lem.gen.rel}, we deduce the following \nameCref{cor.gen.rel}:

\begin{Corollary}\label{cor.gen.rel}
  Let $\lambda\in\Lambda^{+}(n,r)$ and $h\in\Hom_{\k\mathfrak{S}_{r}}(M(\lambda'),M(\lambda))$. Then we have that $h\in\Rel_{\k\mathfrak{S}_{r}}(M(\lambda'),M(\lambda))$ if and only if the coefficients $h[A]$ of the $\rho[A]$ in $h$ satisfy:
    \begin{enumerate}[label=(\roman*), font=\normalfont, ref=(\roman*)]
      \item\label{cor.gen.rel.item.1}For all $1\leq i<j\leq n$, $1\leq k\leq n$, and all $A\in\mathcal{T}_{\lambda}$ with $a_{jk}\neq 0$, we have:
        \begin{equation*}\label{eq.cor.gen.rel.R}
          (a_{ik}+1)h[A]=\sum_{l\neq k}a_{il}h\sq{\rowexop{A}{i}{j}{k}{l}}, \tag{$R_{i,j}^{k}(A)$}
        \end{equation*}
      \item\label{cor.gen.rel.item.2}For all $1\leq i<j\leq n$, $1\leq k\leq n$, and all $A\in\mathcal{T}_{\lambda}$ with $a_{kj}\neq 0$, we have:
        \begin{equation*}\label{eq.cor.gen.rel.C}
          (a_{ki}+1)h[A]=\sum_{l\neq k}a_{li}h\sq{\colexop{A}{i}{j}{k}{l}}. \tag{$C_{i,j}^{k}(A)$}
        \end{equation*}
    \end{enumerate}
\end{Corollary}

\section{A Reduction Trick}\label{sec.trunc.trick}

\subsection{Flattening the partition}\label{subsec.reduc}

Now, we fix integers $a,b,m$ with $a\geq m\geq 2$, and we write $a'\coloneqq b+m-1$, $b'\coloneqq a-m+1$. We denote by $\lambda$ the partition \mbox{$(a,m-1,\dots,2,1^{b})$}, and we fix $r\coloneqq\deg(\lambda)$. Note that the transpose partition $\lambda'$ of $\lambda$ is given by \mbox{$\lambda'=(a',m-1,\dots,2,1^{b'})$}.

\bigskip

Recall that through the ABW-construction of the induced module, we see that $\nabla(\lambda)$ is isomorphic to a \mbox{$G$-quotient} of \mbox{$\Lambda^{\lambda'}E=\Lambda^{a'}E\otimes\Lambda^{m-1}E\otimes\cdots\otimes\Lambda^{2}E\otimes E^{\otimes b'}$}, namely by the submodule $\im\phi_{\lambda'}$ \labelcref{Bu.mapII}. We claim that we can replace the factor $E^{\otimes b'}$ with the symmetric power $S^{b'}E$. This process is in fact independent of the characteristic of the field $\k$. To this end, we construct from the multiplication map \mbox{$\mu:E^{\otimes b'}\to S^{b'}E$}, the surjective \mbox{$G$-homomorphism} \mbox{$1\otimes\mu:\Lambda^{\lambda'}E\to\Lambda^{a'}E\otimes\Lambda^{m-1}E\otimes\cdots\otimes\Lambda^{2}E\otimes S^{b'}E$}.

\begin{Lemma}\label{lem.red.Buch}
  For $m\geq 2$ and $\lambda=(a,m-1,m-2,\ldots,2,1^{b})$, we have:
    \begin{enumerate}[label=(\roman*), font=\normalfont, ref=(\roman*)]
      \item\label{lem.red.Buch.item.1}$\ker(1\otimes\mu)=\sum_{k=1}^{b'-1}\im\smallsup{\phi}[\lambda']{(m+k-1,m+k,1)}\subseteq\im\phi_{\lambda'}$.
      \item\label{lem.red.Buch.item.2}$\nabla(\lambda)\cong\coker((1\otimes\mu)\circ\phi_{\lambda'})$ as $G$-modules.
    \end{enumerate}
\end{Lemma}
\begin{proof}
  \labelcref{lem.red.Buch.item.1} Firstly, that $\im\smallsup{\phi}[\lambda']{(m+k-1,m+k,1)}\subseteq\im\phi_{\lambda'}$ for $1\leq k<b'$ follows from the definition of $\phi_{\lambda'}$. Then, note that by the definition of the symmetric power $S^{b'}E$, the $\k$-space $\ker\mu$ is generated by elements of the form $e_{\bmath{i}}^{[k]}$ for $1\leq k<b'$ and sequences $\bmath{i}\coloneqq(i_{1},\ldots,i_{b'})$ with terms in $[n]$, where $e_{\bmath{i}}^{[k]}\coloneqq(e_{i_{1}}\otimes\cdots\otimes e_{i_{k}}\otimes e_{i_{k+1}}\otimes\cdots\otimes e_{i_{b'}})-(e_{i_{1}}\otimes\cdots\otimes e_{i_{k+1}}\otimes e_{i_{k}}\otimes\cdots\otimes e_{i_{b'}})$. Then, it follows that the $\k$-space $\ker(1\otimes\mu)$ is generated by elements of the form $x\otimes e_{\bmath{i}}^{[k]}$ for $x\in\Lambda^{a'}E\otimes\Lambda^{m-1}E\otimes\cdots\otimes\Lambda^{2}E$, and such $k$ and $\bmath{i}$. But given such $x$, $k$ and $\bmath{i}$, the image of the element $x\otimes e_{i_{1}}\otimes\cdots\otimes(e_{i_{k}}\wedge e_{i_{k+1}})\otimes\cdots\otimes e_{i_{b'}}$ under $\smallsup{\phi}[\lambda']{(m+k-1,m+k,1)}$ is precisely $x\otimes e_{\bmath{i}}^{[k]}$, and so $x\otimes e_{\bmath{i}}^{[k]}\in\im\smallsup{\phi}[\lambda']{(m+k-1,m+k,1)}\subseteq\im\phi_{\lambda'}$, from which \partCref{lem.red.Buch.item.1} follows.

  \labelcref{lem.red.Buch.item.2} Now, the map \mbox{$1\otimes\mu:\Lambda^{\lambda'}E\to\Lambda^{a'}E\otimes\Lambda^{m-1}E\otimes\cdots\otimes\Lambda^{2}E\otimes S^{b'}E$} induces a surjective $G$-homomorphism:
    \[
      \pi:\frac{\Lambda^{\lambda'}E}{\ker(1\otimes\mu)}\to\frac{\Lambda^{a'}E\otimes\Lambda^{m-1}E\otimes\cdots\otimes\Lambda^{2}E\otimes S^{b'}E}{\im((1\otimes\mu)\circ\phi_{\lambda'})}\ .
    \]
  Moreover, it follows from \partCref{lem.red.Buch.item.1} that $\ker\pi=\im\phi_{\lambda'}/\ker(1\otimes\mu)$, and so we deduce that $\nabla(\lambda)\cong\coker((1\otimes\mu)\circ\phi_{\lambda'})$.
\end{proof}

On the other hand, recall that through the James-construction of the induced module, we see that that $\nabla(\lambda)$ is isomorphic to a submodule of $S^{\lambda}E$, namely as the kernel of the \mbox{$G$-homomorphism} $\psi_{\lambda}$ \labelcref{Ja.mapII}. We claim that we may replace the factor $E^{\otimes b}$ with the exterior power $\Lambda^{b}E$. Once again, this process is independent of the characteristic of $\k$. For this, we construct from the comultiplication map $\Delta:\Lambda^{b}E\rightarrow E^{\otimes b}$, the injective \mbox{$G$-homomorphism} \mbox{$1\otimes\Delta:S^{a}E\otimes S^{m-1}E\otimes\cdots\otimes S^{2}E\otimes\Lambda^{b}E\to S^{\lambda}E$}.

\begin{Lemma}\label{lem.red.Jam}
  For $m\geq 2$ and $\lambda=(a,m-1,m-2,\ldots,2,1^{b})$, we have:
    \begin{enumerate}[label=(\roman*), font=\normalfont, ref=(\roman*)]
      \item\label{lem.red.Jam.item.1}$\ker\psi_{\lambda}\subseteq\bigcap_{k=1}^{b-1}\ker\smallsup{\psi}[\lambda]{(m+k-1,m+k,1)}=\im(1\otimes\Delta)$.
      \item\label{lem.red.Jam.item.2}$\nabla(\lambda)\cong\ker(\psi_{\lambda}\circ(1\otimes\Delta))$ as $G$-modules.
    \end{enumerate}
\end{Lemma}
\begin{proof}
  \labelcref{lem.red.Jam.item.1} Firstly, it follows from the definition of $\psi_{\lambda}$ that $\ker\psi_{\lambda}\subseteq\ker\smallsup{\psi}[\lambda]{(m+k-1,m+k,1)}$ for $1\leq k<b$. Then, the \mbox{$\k$-space} $\ker\smallsup{\psi}[\lambda]{(m+k-1,m+k,1)}$ is generated by elements of the form $x\otimes e_{\bmath{i}}^{[k]}$ for $x\in S^{a}E\otimes S^{m-1}E\otimes\cdots\otimes S^{2}E$, $1\leq k<b$, and sequences $\bmath{i}\coloneqq(i_{1},\ldots,i_{b})$ with terms in $[n]$, where $e_{\bmath{i}}^{[k]}\coloneqq(e_{i_{1}}\otimes\cdots\otimes e_{i_{k}}\otimes e_{i_{k+1}}\otimes\cdots\otimes e_{i_{b}})-(e_{i_{1}}\otimes\cdots\otimes e_{i_{k+1}}\otimes e_{i_{k}}\otimes\cdots\otimes e_{i_{b}})$. It follows that the $\k$-space $\bigcap_{k=1}^{b-1}\ker\smallsup{\psi}[\lambda]{(m+k-1,m+k,1)}$ is generated by elements of the form:
    \[
      \sum_{\sigma\in\mathfrak{S}_{b}}\sgn(\sigma)\left(x\otimes e_{i_{\sigma(1)}}\otimes\cdots\otimes e_{i_{\sigma(b)}}\right)=x\otimes\Delta(e_{i_{1}}\wedge\cdots\wedge e_{i_{b}})\in\im(1\otimes\Delta),
    \]
  from which \partCref{lem.red.Jam.item.1} follows.

  \labelcref{lem.red.Jam.item.2} Now, the map $1\otimes\Delta:S^{a}E\otimes S^{m-1}E\otimes\cdots\otimes S^{2}E\otimes\Lambda^{b}E\to S^{\lambda}E$ induces an injective $G$-homomorphism $\nu:\ker(\psi_{\lambda}\circ(1\otimes\Delta))\to\ker\psi_{\lambda}$. Moreover, it follows from \partCref{lem.red.Jam.item.1} that $\nu$ is surjective, and so we have a $G$-isomorphism $\ker(\psi_{\lambda}\circ(1\otimes\Delta))\cong\ker{\psi_{\lambda}}\cong\nabla(\lambda)$.
\end{proof}

Now, we shall return to the situation where the underlying field $\k$ has characteristic $2$. We fix the sequences $\alpha\coloneqq(a',m-1,\ldots,2,b')$ and $\beta\coloneqq(a,m-1,\ldots,2,b)$.

\begin{Remark}\label{rem.ident}
  We shall consider the constructions of this section from the perspective of the Specht module $\Sp(\lambda)$.
    \begin{enumerate}[label=(\roman*), font=\normalfont, ref=(\roman*)]
      \item\label{rem:ident.quo} By \myCref{lem.red.Buch}[lem.red.Buch.item.2] we have that $\nabla(\lambda)\cong\coker((1\otimes\mu)\circ\phi_{\lambda'})$. By applying the Schur functor $f$, we obtain that $\Sp(\lambda)\cong\coker(f(1\otimes\mu)\circ\bar{\phi}_{\lambda'})$. Now, since we are in characteristic $2$, we have that \mbox{$f(\Lambda^{a'}E\otimes\Lambda^{m-1}E\otimes\cdots\otimes\Lambda^{2}E\otimes S^{b'}E)$} is identified with \mbox{$f(S^{a'}E\otimes S^{m-1}E\otimes\cdots\otimes S^{2}E\otimes S^{b'}E)$} which in turn is isomorphic to $M(\alpha)$. We write $\pi_{\alpha}:M(\lambda')\to M(\alpha)$ for the surjective \mbox{$\k\mathfrak{S}_{r}$-homomorphism} that is obtained from $f(1\otimes\mu)$ under these identifications. We set $\bar{\phi}_{\alpha}\coloneqq\pi_{\alpha}\circ\bar{\phi}_{\lambda'}$ and we deduce that $\Sp(\lambda)\cong\coker\bar{\phi}_{\alpha}$.
      \item\label{rem:ident.sub} On the other hand, by \myCref{lem.red.Jam}[lem.red.Jam.item.2] we have that $\nabla(\lambda)\cong\ker(\psi_{\lambda}\circ(1\otimes\Delta))$. By applying the Schur functor $f$, we deduce that $\Sp(\lambda)\cong\ker(\bar{\psi}_{\lambda}\circ f(1\otimes\Delta))$. But once again, since we are in characteristic $2$, \mbox{$f(S^{a}E\otimes S^{m-1}E\otimes\cdots\otimes S^{2}E\otimes\Lambda^{b}E)$} is identified with \mbox{$f(S^{a}E\otimes S^{m-1}E\otimes\cdots\otimes S^{2}E\otimes S^{b}E)$} which in turn is isomorphic to $M(\beta)$. We write $\iota_{\beta}:M(\beta)\to M(\lambda)$ for the injective \mbox{$\k\mathfrak{S}_{r}$-homomorphism} that is obtained from $f(1\otimes\Delta)$ under these identifications. We set $\bar{\psi}_{\beta}\coloneqq\bar{\psi}_{\lambda}\circ\iota_{\beta}$ and we deduce that $\Sp(\lambda)\cong\ker\bar{\psi}_{\beta}$.
    \end{enumerate}
\end{Remark}

We summarise the content of \Cref{rem.ident} in the following \namecref{lem.red.Sp}:

\begin{Lemma}\label{lem.red.Sp}
  For $m\geq 2$ and $\lambda=(a,m-1,m-2,\ldots,2,1^{b})$, we have:
    \begin{enumerate}[label=(\roman*), font=\normalfont, ref=(\roman*)]
      \item\label{lem.red.Sp.item.1}$\Sp(\lambda)\cong\coker\bar{\phi}_{\alpha}$ as \mbox{$\k\mathfrak{S}_{r}$-modules}.
      \item\label{lem.red.Sp.item.2}$\Sp(\lambda)\cong\ker\bar{\psi}_{\beta}$ as \mbox{$\k\mathfrak{S}_{r}$-modules}.
    \end{enumerate}
\end{Lemma}

We define the following \mbox{$\k\mathfrak{S}_{r}$-homomorphisms}:
  \[
    \smallsup{\bar{\phi}}[\alpha]{(i,j,s)}\coloneqq\pi_{\alpha}\circ\smallsup{\bar{\phi}}[\lambda']{(i,j,s)}:M(\smallsup{\lambda'}{(i,j,s)})\to M(\alpha),\quad\smallsup{\bar{\psi}}[\beta]{(i,j,t)}\coloneqq\smallsup{\bar{\psi}}[\lambda]{(i,j,t)}\circ\iota_{\beta}:M(\beta)\to M(\smallsup{\lambda}{(i,j,t)}),
  \]
where $\pi_{\alpha}$ and $\iota_{\beta}$ are as defined in \Cref{rem.ident}. The following properties of these homomorphisms may be easily verified:

\begin{Lemma}\label{lem.vanish}
  For $m\geq 2$ and $\lambda=(a,m-1,m-2,\ldots,2,1^{b})$, we have:
    \begin{enumerate}[label=(\roman*), font=\normalfont, ref=(\roman*)]
      \item\label{lem.vanish.item.1}$\smallsup{\bar{\phi}}[\alpha]{(i,j,s)}=0$ for $m\leq i<j\leq n$, $1\leq s\leq\lambda'_{j}$.
      \item\label{lem.vanish.item.2}$\smallsup{\bar{\psi}}[\beta]{(i,j,t)}=0$ for $m\leq i<j\leq n$, $1\leq t\leq\lambda_{j}$.
      \item\label{lem.vanish.item.3}$\bar{\phi}_{\alpha}=\sum_{i=1}^{m-1}\sum_{s=1}^{\smash[t]{\lambda'_{i+1}}}\smallsup{\bar{\phi}}[\alpha]{(i,i+1,s)}$.
      \item\label{lem.vanish.item.4}$\bar{\psi}_{\beta}=\sum_{i=1}^{m-1}\sum_{t=1}^{\lambda_{i+1}}\smallsup{\bar{\psi}}[\beta]{(i,i+1,t)}$.
    \end{enumerate}
\end{Lemma}

Now, the following \nameCref{lem.flat.map.images} provides an analogue of \Cref{map.images}:

\begin{Lemma}\label{lem.flat.map.images}
  For $m\geq 2$ and $\lambda=(a,m-1,m-2,\ldots,2,1^{b})$, we have:
    \begin{enumerate}[label=(\roman*), font=\normalfont, ref=(\roman*)]
      \item\label{lem.flat.map.images.item.1}$\im\smallsup{\bar{\phi}}[\alpha]{(i,j,s)}\subseteq\im\bar{\phi}_{\alpha}$ for $1\leq i<j\leq m$, $1\leq s\leq\lambda'_{j}$.
      \item\label{lem.flat.map.images.item.2}$\ker\bar{\psi}_{\beta}\subseteq\ker\smallsup{\bar{\psi}}[\beta]{(i,j,t)}$ for $1\leq i<j\leq m$, $1\leq t\leq\lambda_{j}$.
    \end{enumerate}
\end{Lemma}
\begin{proof}
  Firstly, recall the notation, and in particular \mbox{$\k\mathfrak{S}_{r}$-homomorphisms} $\pi_{\alpha}$ and $\iota_{\beta}$, defined within \Cref{rem.ident}. Then, \partCref{lem.flat.map.images.item.1} follows from \myCref{map.images}[map.images.item.1] by applying the Schur functor and post-composing by $\pi_{\alpha}$. Similarly, we see that \partCref{lem.flat.map.images.item.2} follows from \myCref{map.images}[map.images.item.2] by applying the Schur functor and \mbox{pre-composing} by $\iota_{\beta}$.
\end{proof}

Then, by combining the results of \Cref{lem.red.Sp}, \Cref{lem.vanish}, and \Cref{lem.flat.map.images}, we obtain the following description of the endomorphism algebra of $\Sp(\lambda)$:

\begin{Corollary}\label{cor.end.flat}
  The endomorphism algebra of $\Sp(\lambda)$ may be identified with the \mbox{$\k$-subspace} of $\Hom_{\k\mathfrak{S}_{r}}(M(\alpha),M(\beta))$ consisting of those elements $h$ that satisfy:
    \begin{enumerate}[label=(\roman*), font=\normalfont, ref=(\roman*)]
      \item\label{cor.end.flat.item.1}$h\circ\smallsup{\bar{\phi}}[\alpha]{(i,j,s)}=0$ for $1\leq i<j\leq m$ and $1\leq s\leq\lambda'_{j}$,
      \item\label{cor.end.flat.item.2}$\smallsup{\bar{\psi}}[\beta]{(i,j,t)}\circ h=0$ for $1\leq i<j\leq m$ and $1\leq t\leq\lambda_{j}$.
    \end{enumerate}
\end{Corollary}

\begin{Definition}\label{def.sec.rele}
  Let $m\geq 2$, $\lambda=(a,m-1,m-2,\ldots,2,1^{b})$, $\alpha=(a',m-1,\dots,2,b')$, and $\beta=(a,m-1,\ldots,2,b)$. Then:
    \begin{enumerate}[label=(\roman*), font=\normalfont, ref=(\roman*)]
      \item\label{sec.rele.item.1} We say that an element $h\in\Hom_{\k\mathfrak{S}_{r}}(M(\lambda'),M(\lambda))$ is \emph{semirelevant} if $h\circ\smallsup{\phi}[\lambda']{(i,j,1)}=0$ and $\smallsup{\psi}[\lambda]{(i,j,1)}\circ h=0$ for all $m\leq i<j\leq n$.
      \item\label{sec.rele.item.2} We say that an element $h\in\Hom_{\k\mathfrak{S}_{r}}(M(\alpha),M(\beta))$ is \emph{relevant} if $h\circ\smallsup{\bar{\phi}}[\alpha]{(i,j,1)}=0$ and $\smallsup{\bar{\psi}}[\beta]{(i,j,1)}\circ h=0$ for all $1\leq i<j\leq m$.
    \end{enumerate}
\end{Definition}

Denote by $\SRel_{\k\mathfrak{S}_{r}}(M(\lambda'),M(\lambda))$ the $\k$-subspace of $\Hom_{\k\mathfrak{S}_{r}}(M(\lambda'),M(\lambda))$ consisting of the semirelevant homomorphisms $M(\lambda')\to M(\lambda)$, and then, we shall also denote by $\Rel_{\k\mathfrak{S}_{r}}(M(\alpha),M(\beta))$ the $\k$-subspace of $\Hom_{\k\mathfrak{S}_{r}}(M(\alpha),M(\beta))$ consisting of the relevant homomorphisms $M(\alpha)\to M(\beta)$.

\begin{Lemma}\label{lem.semirel.rel}
  Denote by $\omega:\Hom_{\k\mathfrak{S}_{r}}(M(\alpha),M(\beta))\to\Hom_{\k\mathfrak{S}_{r}}(M(\lambda'),M(\lambda))$ the $\k$-linear homomorphism with $\omega(h)\coloneqq\iota_{\beta}\circ h\circ\pi_{\alpha}$. Then $\omega$ induces the following $\k$-linear isomorphisms:
    \begin{enumerate}[label=(\roman*), font=\normalfont, ref=(\roman*)]
      \item\label{lem.semirel.rel.item.1}$\hat{\omega}: \Hom_{\k\mathfrak{S}_{r}}(M(\alpha),M(\beta))\to\SRel_{\k\mathfrak{S}_{r}}(M(\lambda'),M(\lambda))$.
      \item\label{lem.semirel.rel.item.2}$\bar{\omega}: \Rel_{\k\mathfrak{S}_{r}}(M(\alpha),M(\beta))\to\Rel_{\k\mathfrak{S}_{r}}(M(\lambda'),M(\lambda))$.
    \end{enumerate}
\end{Lemma}
\begin{proof}
  The maps $\hat{\omega}$ and $\bar{\omega}$ are clearly injective. Now, \myCref{lem.red.Buch}[lem.red.Buch.item.1] and \myCref{lem.red.Jam}[lem.red.Jam.item.1] give that both maps are surjective.
\end{proof}

\begin{Remark}\label{rem.comp}
  Let $\gamma\in\Lambda(n,r)$ with $\ell\coloneqq\ell(\gamma)$. Then:
    \begin{enumerate}[label=(\roman*), font=\normalfont, ref=(\roman*)]
      \item\label{rem.comp.item.1} Fix $B\in\Tab(\alpha,\gamma)$. Then $\rho[B]\circ\pi_{\alpha}\in\Hom_{\k\mathfrak{S}_{r}}(M(\lambda'),M(\gamma))$ and one can easily check that $\rho[B]\circ\pi_{\alpha}=\sum_{A}\rho[A]$, where the sum is over those $A\in\Tab(\lambda',\gamma)$ whose first $(m-1)$ rows agree with those of $B$, and also $\sum_{i=m}^{a}a_{ij}=b_{mj}$ for $1\leq j\leq\ell$. Informally, these $A$ are obtained from $B$ by distributing, along columns, each \mbox{non-zero} entry within the $m$th-row of $B$ into rows $m$ through $a$ of $A$ such that these rows of $A$ contain exactly one non-zero, and hence equal to $1$, entry.

      \item\label{rem.comp.item.2} Now, let $B\in\Tab(\gamma,\beta)$. Then $\iota_{\beta}\circ\rho[B]\in\Hom_{\k\mathfrak{S}_{r}}(M(\gamma),M(\lambda))$ and one can easily check that $\iota_{\beta}\circ\rho[B]=\sum_{A}\rho[A]$, where the sum is over those $A\in\Tab(\gamma,\lambda)$ whose first $(m-1)$ columns agree with those of $B$, and also $\sum_{j=m}^{\smash[t]{a'}}a_{ij}=b_{im}$ for $1\leq i\leq\ell$. Informally, these $A$ are obtained from $B$ by distributing, along rows, each \mbox{non-zero} entry within the $m$th-column of $B$ into columns $m$ through $a'$ of $A$ such that these columns of $A$ contain exactly one non-zero, and hence equal to $1$, entry.
    \end{enumerate}
\end{Remark}

\begin{Example}\label{ex.distribute}
  For $\lambda=(3,1^{3})$, we have:
    \begin{alignat*}{2}
      \rho\left[\ \begin{matrix}
        2&2 \\
        1&1
      \end{matrix}\ \right]\circ\pi_{(4,2)}
      =&\ \rho\left[\ \begin{matrix}%
        2&2 \\
        1&0 \\
        0&1
      \end{matrix}\ \right]+\rho\left[\ \begin{matrix}%
        2&2 \\
        0&1 \\
        1&0
      \end{matrix}\ \right], \\
    \iota_{(3,3)}\circ\rho\left[\ \begin{matrix}%
        2&2 \\
        1&1
      \end{matrix}\ \right]
      =&\ \rho\left[\ \begin{matrix}
        2&1&1&0 \\
        1&0&0&1
      \end{matrix}\ \right]+\rho\left[\ \begin{matrix}%
        2&1&0&1\\
        1&0&1&0
      \end{matrix}\ \right]+\rho\left[\ \begin{matrix}%
        2&0&1&1\\
        1&1&0&0
      \end{matrix}\ \right], \\
      \iota_{(3,3)}\circ\rho\left[\ \begin{matrix}%
        2&2 \\
        1&1
      \end{matrix}\ \right]\circ\pi_{(4,2)}
      =&\ \rho\left[\ \begin{matrix}
        2&1&1&0 \\
        1&0&0&0 \\
        0&0&0&1
      \end{matrix}\ \right]+\rho\left[\ \begin{matrix}%
        2&1&0&1 \\
        1&0&0&0 \\
        0&0&1&0
      \end{matrix}\ \right]+\rho\left[\ \begin{matrix}%
        2&0&1&1 \\
        1&0&0&0 \\
        0&1&0&0
      \end{matrix}\ \right] \\
      +&\ \rho\left[\ \begin{matrix}
        2&1&1&0 \\
        0&0&0&1 \\
        1&0&0&0
      \end{matrix}\ \right]+\rho\left[\ \begin{matrix}%
        2&1&0&1 \\
        0&0&1&0 \\
        1&0&0&0
      \end{matrix}\ \right]+\rho\left[\ \begin{matrix}%
        2&0&1&1 \\
        0&1&0&0 \\
        1&0&0&0
      \end{matrix}\ \right].
    \end{alignat*}
\end{Example}

The following \nameCref{lem.end.red} provides an analogue of \Cref{cor.gen.rel}:

\begin{Lemma}\label{lem.end.red}
  Let $h\in\Hom_{\k\mathfrak{S}_{r}}(M(\alpha),M(\beta))$. Then \mbox{$h\in\Rel_{\k\mathfrak{S}_{r}}(M(\alpha),M(\beta))$} if and only if the coefficients $h[B]$ of the $\rho[B]$ in $h$ satisfy:
    \begin{enumerate}[label=(\roman*), font=\normalfont, ref=(\roman*)]
      \item\label{lem.end.red.item.1}For all $1\leq i<j\leq m$, $1\leq k\leq m$, and all $B\in\Tab(\alpha,\beta)$ with $b_{jk}\neq 0$, we have:
        \begin{equation*}\label{eq.end.red.R}
          (b_{ik}+1)h[B]=\sum_{l\neq k}b_{il}h\sq{\rowexop{B}{i}{j}{k}{l}}, \tag{$R_{i,j}^{k}(B)$}
        \end{equation*}
      \item\label{lem.end.red.item.2}For all $1\leq i<j\leq m$, $1\leq k\leq m$, and all $B\in\Tab(\alpha,\beta)$ with $b_{kj}\neq 0$, we have:
        \begin{equation*}\label{eq.end.red.C}
          (b_{ki}+1)h[B]=\sum_{l\neq k}b_{li}h\sq{\colexop{B}{i}{j}{k}{l}}. \tag{$C_{i,j}^{k}(B)$}
        \end{equation*}
    \end{enumerate}
\end{Lemma}
\begin{proof}
  For $B\in\Tab(\alpha,\beta)$, we denote by $\Omega(B)$ the subset of matrices in $\Tab(\lambda',\lambda)$ with:
    \begin{equation}\label{eq.end.red.0}
      \omega(\rho[B])=\iota_{\beta}\circ\rho[B]\circ\pi_{\alpha}=\sum_{A\in\Omega(B)}\rho[A].
    \end{equation}
  Now, we fix $h\in\Hom_{\k\mathfrak{S}_{r}}(M(\alpha),M(\beta))$ with $h=\sum_{B\in\Tab(\alpha,\beta)}h[B]\rho[B]$, and we shall set $\tilde{h}\coloneqq\omega(h)=\iota_{\beta}\circ h\circ\pi_{\alpha}$. Then, it follows from \Cref{rem.comp} that the coefficients $\tilde{h}[A]$ of the $\rho[A]$ in $\tilde{h}$ satisfy:
    \begin{equation}\label{eq.end.red.1}
      \tilde{h}[A]=
        \begin{cases}
          h[B],&\text{if}\ A\in\Omega(B)\settext{for some}B\in\Tab(\alpha,\beta), \\
          0,&\text{otherwise.}
        \end{cases}
    \end{equation}
  Now, suppose that $h$ is relevant and we shall show that the coefficients $h[B]$ of the $\rho[B]$ in $h$ satisfy the relations stated in \labelcref{lem.end.red.item.1}, and it may be shown in a similar manner that they also satisfy the relations stated in \labelcref{lem.end.red.item.2}. Firstly, note that $\tilde{h}$ is relevant by \myCref{lem.semirel.rel}[lem.semirel.rel.item.2]. We fix $1\leq i<j\leq m$, $1\leq k\leq m$, and $B\in\Tab(\alpha,\beta)$ with $b_{jk}\neq 0$. Then, there exists $A\in\Omega(B)$ with $a_{jk}\neq 0$. For such an $A$, since $\tilde{h}$ is relevant, the relation \hyperref[eq.cor.gen.rel.R]{$R_{i,j}^{k}(A)$} of \myCref{cor.gen.rel}[cor.gen.rel.item.2] gives that:
    \begin{equation}\label{eq.end.red.2}
      (a_{ik}+1)\tilde{h}[A]=\sum_{l\neq k}a_{il}\tilde{h}\sq{\rowexop{A}{i}{j}{k}{l}}.
    \end{equation}
  Now, take any $1\leq l\leq n$ with $l\neq k$ such that $a_{il}\neq 0$. If $l<m$, then $a_{il}=b_{il}$ and $\rowexop{A}{i}{j}{k}{l}\in\Omega(\rowexop{B}{i}{j}{k}{l})$, so that $\tilde{h}\sq{\rowexop{A}{i}{j}{k}{l}}=h\sq{\rowexop{B}{i}{j}{k}{l}}$. On the other hand, if $l\geq m$, then $a_{il}=1$ with $\rowexop{A}{i}{j}{k}{l}\in\Omega(\rowexop{B}{i}{j}{k}{m})$ so that $\tilde{h}\sq{\rowexop{A}{i}{j}{k}{l}}=h\sq{\rowexop{B}{i}{j}{k}{m}}$. Therefore, we may rewrite \labelcref{eq.end.red.2} as:
    \begin{equation}\label{eq.end.red.3}
      (a_{ik}+1)h[B]=\sum_{\substack{l<m\\l\neq k}}b_{il}h\sq{\rowexop{B}{i}{j}{k}{l}}+\Big(\sum_{\substack{l\geq m\\l\neq k}}a_{il}\Big)h\sq{\rowexop{B}{i}{j}{k}{m}}.
    \end{equation}
  Now, if $k<m$, then $a_{ik}=b_{ik}$ and $\sum_{l\geq m}a_{il}=b_{im}$. Thus, \labelcref{eq.end.red.3} becomes:
    \[
      (b_{ik}+1)h[B]=\sum_{\substack{l<m\\l\neq k}}b_{il}h\sq{\rowexop{B}{i}{j}{k}{l}}+b_{im}\sq{\rowexop{B}{i}{j}{k}{m}}=\sum_{l\neq k}b_{il}h\sq{\rowexop{B}{i}{j}{k}{l}},
    \]
  which is precisely the relation \hyperref[eq.end.red.R]{$R_{i,j}^{k}(B)$}.

  On the other hand, if $k=m$, then $a_{im}=0$, since $a_{jm}\neq 0$, and so $\sum_{l>m}a_{il}=b_{im}$. Moreover, $\rowexop{B}{i}{j}{k}{m}=B$, and so \labelcref{eq.end.red.3} becomes:
    \[
      h[B]=\sum_{l<m}b_{il}h\sq{\rowexop{B}{i}{j}{k}{l}}+b_{im}h[B],
    \]
  which in turn gives the relation \hyperref[eq.end.red.R]{$R_{i,j}^{m}(B)$}:
    \[
      (b_{im}+1)h[B]=\sum_{l\neq m}b_{il}h\sq{\rowexop{B}{i}{j}{k}{l}}.
    \]
  Conversely, suppose that the coefficients $h[B]$ of the $\rho[B]$ in $h$ satisfy the relations stated in the \nameCref{lem.end.red}. Note that by \myCref{lem.semirel.rel}[lem.semirel.rel.item.2], in order to show that $h$ is relevant, it suffices to show that $\tilde{h}$ is relevant. To this end, we shall show that $\tilde{h}\circ\smallsup{\bar{\phi}}[\lambda']{(i,j,1)}=0$ for $1\leq i<j\leq n$, and it shall follow similarly that $\smallsup{\bar{\psi}}[\lambda]{(i,j,1)}\circ\tilde{h}=0$ for such $i$, $j$. Note that $\tilde{h}$ is semirelevant by \myCref{lem.semirel.rel}[lem.semirel.rel.item.1] and so $\tilde{h}\circ\smallsup{\bar{\phi}}[\lambda']{(i,j,1)}=0$ for $i\geq m$. Therefore, we may assume that $i<m$. Accordingly, fix some $1\leq i<j\leq n$ with $i<m$. Then, as in the proof of \Cref{lem.gen.rel}, we have:
    \begin{equation}\label{eq.end.red.4}
      \tilde{h}\circ\smallsup{\bar{\phi}}[\lambda']{(i,j,1)}=\sum_{C\in\Tab(\smallsup{\lambda'}{(i,j,1)},\lambda)}\left(\sum_{1\leq l\leq n}c_{il}\tilde{h}\sq{\smallsubsup{C}{(i,l)}{(j,l)}}\right)\rho[C].
    \end{equation}
  Let $C\in\Tab(\smallsup{\lambda'}{(i,j,1)},\lambda)$, and we wish to show that the coefficient of $\rho[C]$ in $\tilde{h}\circ\smallsup{\bar{\phi}}[\lambda']{(i,j,1)}$ is equal to $0$. According to \labelcref{eq.end.red.1} and \labelcref{eq.end.red.4}, we may assume that there exists some $1\leq k\leq n$ with $c_{ik}\neq 0$ such that $A\coloneqq\smallsubsup{C}{(i,k)}{(j,k)}\in\Omega(B)$ for some $B\in\Tab(\alpha,\beta)$, where $\Omega(B)$ is as in \labelcref{eq.end.red.0}, since otherwise, each $\tilde{h}\sq{\smallsubsup{C}{(i,l)}{(j,l)}}$ appearing in the coefficient of $\rho[C]$ in \labelcref{eq.end.red.4} is equal to zero. Then, it follows from \labelcref{eq.end.red.4} that the coefficient of $\rho[C]$ in $\tilde{h}\circ\smallsup{\bar{\phi}}[\lambda']{(i,j,1)}$ is:
    \begin{equation}\label{eq.end.red.5}
      c_{ik}h[B]+\sum_{\substack{1\leq l\leq n\\l\neq k}}c_{il}\tilde{h}\sq{\rowexop{A}{i}{j}{k}{l}}.
    \end{equation}
  We split our consideration into the following cases:
    \begin{enumerate}[label=(\roman*), font=\normalfont, ref=(\roman*)]
      \item\label{lem.end.red.proof.item.1}$(j<m;\ k<m)$: We have $c_{ik}=a_{ik}+1=b_{ik}+1$. Now, if $1\leq l<m$ with $l\neq k$, then $c_{il}=a_{il}=b_{il}$ with $\rowexop{A}{i}{j}{k}{l}\in\Omega(\rowexop{B}{i}{j}{k}{l})$ so that $\tilde{h}\sq{\rowexop{A}{i}{j}{k}{l}}=h\sq{\rowexop{B}{i}{j}{k}{l}}$. On the other hand, if $l\geq m$ with $c_{il}\neq 0$, then $c_{il}=a_{il}=1$ with $\rowexop{A}{i}{j}{k}{l}\in\Omega(\rowexop{B}{i}{j}{k}{m})$ so that $\tilde{h}\sq{\rowexop{A}{i}{j}{k}{l}}=h\sq{\rowexop{B}{i}{j}{k}{m}}$. Note that there are precisely $b_{im}$ such values of $l$. Hence, we may rewrite \labelcref{eq.end.red.5} as:
        \[
          (b_{ik}+1)h[B]+\sum_{\substack{1\leq l<m\\l\neq k}}b_{il}h\sq{\rowexop{B}{i}{j}{k}{l}}+b_{im}h\sq{\rowexop{B}{i}{j}{k}{m}}=0,
        \]
      since the coefficient $h[B]$ satisfies the relation \hyperref[eq.end.red.R]{$R_{i,j}^{k}(B)$}.
      \item\label{lem.end.red.proof.item.2}$(j<m;\ k\geq m)$: Here, we have $c_{ik}=1$ and also $b_{jm}\neq 0$ since $A\in\Omega(B)$. Now, if $1\leq l<m$, then $c_{il}=a_{il}=b_{il}$ with $\rowexop{A}{i}{j}{k}{l}\in\Omega(\rowexop{B}{i}{j}{m}{l})$ so that $\tilde{h}\sq{\rowexop{A}{i}{j}{k}{l}}=h\sq{\rowexop{B}{i}{j}{m}{l}}$. On the other hand, if $l\geq m$ with $l\neq k$ and $c_{il}\neq 0$, then $c_{il}=a_{il}=1$ with $\rowexop{A}{i}{j}{k}{l}\in\Omega(B)$ so that $\tilde{h}\sq{\rowexop{A}{i}{j}{k}{l}}=h[B]$. Note that there are precisely $b_{im}$ such values of $l$. Hence, we may rewrite \labelcref{eq.end.red.5} as:
        \[
          h[B]+\sum_{1\leq l<m}b_{il}h\sq{\rowexop{B}{i}{j}{m}{l}}+b_{im}h[B]=0,
        \]
      since the coefficient $h[B]$ satisfies the relation \hyperref[eq.end.red.R]{$R_{i,j}^{m}(B)$}.
      \item\label{lem.end.red.proof.item.3}$(j\geq m;\ k<m)$: Now, we have $c_{ik}=a_{ik}+1=b_{ik}+1$ and also $b_{mk}\neq 0$ since $A\in\Omega(B)$. Now, if $1\leq l<m$ with $l\neq k$, then $c_{il}=a_{il}=b_{il}$ with $\rowexop{A}{i}{j}{k}{l}\in\Omega(\rowexop{B}{i}{m}{k}{l})$ so that $\tilde{h}\sq{\rowexop{A}{i}{j}{k}{l}}=h\sq{\rowexop{B}{i}{m}{k}{l}}$. On the other hand, if $l\geq m$ with $c_{il}\neq 0$, then $c_{il}=a_{il}=1$ with $\rowexop{A}{i}{j}{k}{l}\in\Omega(\rowexop{B}{i}{m}{k}{m})$ so that $\tilde{h}\sq{\rowexop{A}{i}{j}{k}{l}}=h\sq{\rowexop{B}{i}{m}{k}{m}}$. Note that there are precisely $b_{im}$ such values of $l$. Hence, we may rewrite \labelcref{eq.end.red.5} as:
        \[
          (b_{ik}+1)h[B]+\sum_{\substack{1\leq l<m\\l\neq k}}b_{il}h\sq{\rowexop{B}{i}{m}{k}{l}}+b_{im}h\sq{\rowexop{B}{i}{m}{k}{m}}=0,
        \]
      since the coefficient $h[B]$ satisfies the relation \hyperref[eq.end.red.R]{$R_{i,m}^{k}(B)$}.
      \item\label{lem.end.red.proof.item.4}$(j\geq m;\ k\geq m)$: Finally, in this case, we have $c_{ik}=1$ and also $b_{mm}\neq 0$ since $A\in\Omega(B)$. Now, if $1\leq l<m$, then $c_{il}=a_{il}=b_{il}$ with $\rowexop{A}{i}{j}{k}{l}\in\Omega(\rowexop{B}{i}{m}{m}{l})$ so that $\tilde{h}\sq{\rowexop{A}{i}{j}{k}{l}}=h\sq{\rowexop{B}{i}{m}{m}{l}}$. On the other hand, if $l\geq m$ with $l\neq k$ and $c_{il}\neq 0$, then $c_{il}=a_{il}=1$ with $\rowexop{A}{i}{j}{k}{l}\in\Omega(B)$ so that $\tilde{h}\sq{\rowexop{A}{i}{j}{k}{l}}=h[B]$. Note that there are precisely $b_{im}$ such values of $l$. Hence, we may rewrite \labelcref{eq.end.red.5} as:
        \[
          h[B]+\sum_{1\leq l<m}b_{il}h\sq{\rowexop{B}{i}{m}{m}{l}}+b_{im}h[B]=0,
        \]
      since the coefficient $h[B]$ satisfies the relation \hyperref[eq.end.red.R]{$R_{i,m}^{m}(B)$}.
    \end{enumerate}
  Thus, we have shown that the coefficient of $\rho[C]$ in $\tilde{h}\circ\smallsup{\bar{\phi}}[\lambda']{(i,j,1)}$ is zero in all possible cases, and so we are done.
\end{proof}

Now, since $\alpha$ and $\beta$ both have length $m$, we may ignore the final $(n-m)$ rows and columns of each matrix in $\Tab(\alpha,\beta)$ and $\Tab(\beta,\alpha)$. Accordingly, we identify $\Tab(\alpha,\beta)$ with the set $\mathcal{T}\coloneqq\{A\in M_{m\times m}(\mathbb{N})\mid\sum_{j}a_{ij}=\alpha_{i}\settext{and}\sum_{i}a_{ij}=\beta_{j}\}$, and $\Tab(\beta,\alpha)$ with the set $\mathcal{T}'\coloneqq\{A\in M_{m\times m}(\mathbb{N})\mid\sum_{j}a_{ij}=\beta_{i}\settext{and}\sum_{i}a_{ij}=\alpha_{j}\}$.

\begin{Remark}\label{rem.swap}
  Note that $\lambda$ and its transpose $\lambda'$ are of the same form. That is to say, the swap $\lambda\leftrightarrow\lambda'$ is equivalent to the swap $(a,b)\leftrightarrow(a',b')$, where \mbox{$a'=b+m-1$}, \mbox{$b'=a-m+1$} respectively, which in turn is equivalent to the swap $\alpha\leftrightarrow\beta$. Therefore, after defining the notion of \emph{relevance} for elements $h\in\Hom_{\k\mathfrak{S}_{r}}(M(\beta),M(\alpha))$, similarly to \myCref{def.sec.rele}[sec.rele.item.2], and also swapping $\mathcal{T}$ with $\mathcal{T}'$, we obtain the following analogue of \Cref{lem.end.red}:
\end{Remark}

\begin{Corollary}\label{cor.end.red.tr}
  Let $h\in\Hom_{\k\mathfrak{S}_{r}}(M(\beta),M(\alpha))$. Then \mbox{$h\in\Rel_{\k\mathfrak{S}_{r}}(M(\beta),M(\alpha))$} if and only if the coefficients $h[B]$ of the $\rho[B]$ in $h$ satisfy:
    \begin{enumerate}[label=(\roman*), font=\normalfont, ref=(\roman*)]
      \item\label{cor.end.red.tr.item.1}$R_{i,j}^{k}(B)$ for all $1\leq i<j\leq m$, $1\leq k\leq m$, and $B\in\mathcal{T}'$ with $b_{jk}\neq 0$,
      \item\label{cor.end.red.tr.item.2}$C_{i,j}^{k}(B)$ for all $1\leq i<j\leq m$, $1\leq k\leq m$, and $B\in\mathcal{T}'$ with $b_{kj}\neq 0$.
    \end{enumerate}
\end{Corollary}

The following \nameCref{rem.ide.end} is clear:

\begin{Remark}\label{rem.ide.end}
  Let $m\geq 2$ and $\lambda=(a,m-1,m-2,\ldots,2,1^{b})$. Then:
    \begin{enumerate}[label=(\roman*), font=\normalfont, ref=(\roman*)]\label{rem.emb.second.rel}
      \item\label{rem.emb.second.rel.item.1} We have a $\k$-linear embedding of the endomorphism algebra of $\Sp(\lambda)$ into the \mbox{$\k$-space} $\Rel_{\k\mathfrak{S}_{r}}(M(\alpha),M(\beta))$.
      \item\label{rem.emb.second.rel.item.2} We have a $\k$-linear embedding of the endomorphism algebra of $\Sp(\lambda')$ into the \mbox{$\k$-space} $\Rel_{\k\mathfrak{S}_{r}}(M(\beta),M(\alpha))$.
    \end{enumerate}
\end{Remark}

\begin{Remark}\label{lem.hom.rel}
  Let $h\in\Hom_{\k\mathfrak{S}_{r}}(M(\alpha),M(\beta))$ and consider its transpose homomorphism $h'\in\Hom_{\k\mathfrak{S}_{r}}(M(\beta),M(\alpha))$. We have:
    \begin{enumerate}[label=(\roman*), font=\normalfont, ref=(\roman*)]
      \item\label{lem.hom.rel.item.1}For $1\leq i<j\leq m$, $1\leq k\leq m$, and $A\in\mathcal{T}$ with $a_{jk}\neq 0$, the relation $R_{i,j}^{k}(A)$ concerning the coefficient of $\rho[A]$ in $h$ coincides with the relation $C_{i,j}^{k}(A')$ concerning the coefficient of $\rho[A']$ in $h'$.
      \item\label{lem.hom.rel.item.2}For $1\leq i<j\leq m$, $1\leq k\leq m$, and $A\in\mathcal{T}$ with $a_{kj}\neq 0$, the relation $C_{i,j}^{k}(A)$ concerning the coefficient of $\rho[A]$ in $h$ coincides with the relation $R_{i,j}^{k}(A')$ concerning the coefficient of $\rho[A']$ in $h'$.
      \item\label{lem.hom.rel.item.3}The transpose homomorphism $h'$ is relevant if and only if $h$ is relevant.
    \end{enumerate}
\end{Remark}

\subsection{A critical relation}\label{subsec.new.rel}

Here, we shall highlight a new relation that occurs as a combination of the relations $R_{i,j}^{k}(A)$ and $C_{i,j}^{k}(A)$ of \Cref{lem.end.red} that will play an important role in our considerations below.

\begin{Lemma}\label{lem.critical.relation}
  Suppose that $h\in\Hom_{\k\mathfrak{S}_{r}}(M(\alpha),M(\beta))$ is a relevant homomorphism. Then the coefficients $h[A]$ of the $\rho[A]$ in $h$ satisfy the relations:
    \begin{equation*}\label{eq.critical.relation}
      z_{j,k}(A)h[A]=\sum_{\substack{i<j\\l>k}}a_{il}h\sq{\rowexop{A}{i}{j}{k}{l}}+\sum_{\substack{i>j\\l<k}}a_{il}h\sq{\rowexop{A}{i}{j}{k}{l}}, \tag{$Z_{j,k}(A)$}
    \end{equation*}
  for all $1\leq j,k\leq m$ and $A\in\mathcal{T}$ with $a_{jk}\neq 0$, where $z_{j,k}(A)\coloneqq\sum\limits_{i<j}a_{ik}+\sum\limits_{l<k}a_{jl}+j+k\in\k$.
\end{Lemma}
\begin{proof}
  Since $h$ is relevant, the coefficients $h[A]$ of the $\rho[A]$ in $h$ satisfy the relations of \Cref{lem.end.red}, and so in particular, given $1\leq j,k\leq m$, the coefficients satisfy the relation $\sum_{i<j}R_{i,j}^{k}(A)+\sum_{l<k}C_{l,k}^{j}(A)$ for all $A\in\mathcal{T}$ with $a_{jk}\neq 0$. But, the \mbox{left-hand} side of this relation is given by:
    \begin{equation}\label{eq.lem.proof.critical.relation.1}
      \sum_{i<j}(a_{ik}+1)h[A]+\sum_{l<k}(a_{jl}+1)h[A]=z_{j,k}(A)h[A],
    \end{equation}
  by definition of $z_{j,k}(A)$. On the other hand, the \mbox{right-hand} side of this relation is:
    \begin{equation}\label{eq.lem.proof.critical.relation.2}
      \sum_{\substack{i<j\\l\neq k}}a_{il}h\sq{\rowexop{A}{i}{j}{k}{l}}+\sum_{\substack{l<k\\i\neq j}}a_{il}h\sq{\colexop{A}{l}{k}{j}{i}}.
    \end{equation}
  Now, notice that for $i<j$, $l<k$ we have $\colexop{A}{l}{k}{j}{i}=\smallsubsup{A}{(j,k)(i,l)}{(i,k)(j,l)}$ and so after cancelling those terms that appear twice, we may rewrite \labelcref{eq.lem.proof.critical.relation.2} as:
    \[
      \sum_{\substack{i<j\\l\neq k}}a_{il}h\sq{\rowexop{A}{i}{j}{k}{l}}+\sum_{\substack{l<k\\i\neq j}}a_{il}h\sq{\rowexop{A}{i}{j}{k}{l}}=\sum_{\substack{i<j\\l>k}}a_{il}h\sq{\rowexop{A}{i}{j}{k}{l}}+\sum_{\substack{i>j\\l<k}}a_{il}h\sq{\rowexop{A}{i}{j}{k}{l}},
    \]
 which, along with \labelcref{eq.lem.proof.critical.relation.1}, gives the required expression.
\end{proof}

\section{One-dimensional endomorphism algebra}\label{sec.end.alg.main}

Given integers $s$, $t$, we write $s \equiv t$ to mean that $s$ is congruent to $t$ modulo $2$, and so in particular, are equal as elements of the field $\k$. From here, we shall assume that the parameters $a$, $b$, and $m$ satisfy the parity condition: \mbox{$a-m\equiv b\mod{2}$}. Note that this condition is preserved by the swap \mbox{$(a,b)\leftrightarrow(a',b')$}, where $a'=b+m-1$, $b'=a-m+1$.

\bigskip

Firstly, we highlight some basic properties of the coefficients $z_{j,k}(A)$ from \Cref{lem.critical.relation}.

\begin{Lemma}\label{lem.critical.rel.coef}
  Let $A\in\mathcal{T}$. Then:
    \begin{enumerate}[label=(\roman*), font=\normalfont, ref=(\roman*)]
      \item\label{lem.critical.rel.coef.item.1}$z_{j,k}(A)=\sum_{i>j}a_{ik}+\sum_{l>k}a_{jl}+\alpha_{j}+\beta_{k}+j+k$ for $1\leq j,k\leq m$.
      \item\label{lem.critical.rel.coef.item.2}$z_{j,k}(A)=\sum_{i>j}a_{ik}+\sum_{l>k}a_{jl}$ for $1<j,k<m$.
      \item\label{lem.critical.rel.coef.item.3}$z_{j,m}(A)=b+1+\sum_{i>j}a_{im}$ and $z_{m,k}(A)=a+m+\sum_{i>k}a_{mi}$ for $1<j,k<m$.
      \item\label{lem.critical.rel.coef.item.4}$z_{m,m}(A)=1$.
      \item\label{lem.critical.rel.coef.item.5}$z_{1,m}(A)=\sum_{i>1}a_{im}$ and $z_{m,1}(A)=\sum_{i>1}a_{mi}$.
    \end{enumerate}
\end{Lemma}
\begin{proof}
  \PartCref{lem.critical.rel.coef.item.1} follows from substituting the two expressions: $\sum_{i<j}a_{ik}=\beta_{k}-\sum_{i\geq j}a_{ik}$ and $\sum_{l<k}a_{jl}=\alpha_{j}-\sum_{l\geq k}a_{ji}$ into the definition of $z_{i,j}(A)$. \hyperref[lem.critical.rel.coef]{Parts \labelcref*{lem.critical.rel.coef.item.2}-\labelcref*{lem.critical.rel.coef.item.5}} then follow immediately from \partCref{lem.critical.rel.coef.item.1} along with the forms of $\alpha$ and $\beta$.
\end{proof}

\begin{Definition}\label{def.order}
  Let $A, B\in\mathcal{T}$. Then:
    \begin{enumerate}[label=(\roman*), font=\normalfont, ref=(\roman*)]
      \item\label{def.order.item.1}We write $A<_{R}B$ to mean that $B$ follows $A$ under the induced lexicographical order on rows, reading left to right and bottom to top. This is a total order and we call it the \emph{row-order}.
      \item\label{def.order.item.2}We write $A<_{C}B$ to mean that $B$ follows $A$ under the induced lexicographical order on columns, reading top to bottom and right to left. This is a total order and we call it the \emph{column-order}.
    \end{enumerate}
\end{Definition}

\begin{Remark}\label{rem.order}
  Let $1\leq j,k\leq m$ and let $A\in\mathcal{T}$ with $a_{jk}\neq 0$. Then any $B=\smallsubsup{A}{(j,k)(i,l)}{(i,k)(j,l)}$ that appears in the relation \hyperref[eq.critical.relation]{$Z_{j,k}(A)$} of \Cref{lem.critical.relation} satisfies both $B<_{R}A$ and $B<_{C}A$.
\end{Remark}

From now on, we fix a relevant homomorphism $h\in\Hom_{\k\mathfrak{S}_{r}}(M(\alpha),M(\beta))$.

\begin{Lemma}\label{lem.bottom.right.entry}
  Let $A\in\mathcal{T}$ and suppose that $a_{mm}\neq 0$. Then $h[A]=0$.
\end{Lemma}
\begin{proof}
  Firstly, $z_{m,m}(A)=1$ by \myCref{lem.critical.rel.coef}[lem.critical.rel.coef.item.4], and the result follows by \hyperref[eq.critical.relation]{$Z_{m,m}(A)$}. \qedhere
\end{proof}

\begin{Remark}\label{rem.Murphy}
  Assume that $m=2$, where then $\alpha=(b+1,a-1)$ and $\beta=(a,b)$. Suppose that $h\in\Hom_{\k\mathfrak{S}_{r}}(M(\alpha),M(\beta))$ is a non-zero relevant homomorphism, and suppose that $A\in\mathcal{T}$ is such that $h[A]\neq 0$. We may assume that $a_{22}=0$ by \Cref{lem.bottom.right.entry}. Now, since $a_{12}+a_{22}=b$ and $a_{21}+a_{22}=a-1$, we deduce that $a_{12}=b$ and $a_{21}=a-1$. Moreover, since $a_{11}+a_{12}=b+1$, we have that $a_{11}=1$. Hence, there is a unique matrix $A$ for which $h[A]\neq 0$, namely:
    \[
      A=
        \begin{tabular}{|c|c|}
          \hline
            $1$ & $b$ \\
          \hline
            $a-1$ &$0$ \\
          \hline
        \end{tabular}
      \ .
    \]
  Hence for $\lambda=(a,1^{b})$ with $a\equiv b\mod{2}$, we deduce that $\End_{\k\mathfrak{S}_{r}}(\Sp(\lambda))\cong\k$, and in this way we recover Murphy's result \cite[Theorem 4.1]{M}.
\end{Remark}

\begin{Lemma}\label{lem.outside.rim}
  Let $A\in\mathcal{T}$ and suppose that there exist some $1<j,k<m$ such that $a_{jm}\neq 0$ and $a_{mk}\neq 0$. Then $h[A]=0$.
\end{Lemma}
\begin{proof}
  Suppose for contradiction that the claim is false and let $A\in\mathcal{T}$ be a counterexample that is minimal with respect to the column-order $<_{C}$. We choose $1<j,k<m$ to be maximal such that $a_{jm},a_{mk}\neq 0$. We may assume that $a_{mm}=0$ by \Cref{lem.bottom.right.entry}. Now, by \myCref{lem.critical.rel.coef}[lem.critical.rel.coef.item.3] we have $z_{j,m}(A)+z_{m,k}(A)=1$ and so the relation $Z_{j,m}(A)+Z_{m,k}(A)$ gives:
    \[
      h[A]=\sum_{\substack{i>j\\l<m}}a_{il}h[\lowsmallsup{B}{[i,l]}]+\sum_{\substack{i<m\\l>k}}a_{il}h[\lowsmallsup{D}{[i,l]}],
    \]
  where $\lowsmallsup{B}{[i,l]}\coloneqq\rowexop{A}{i}{j}{m}{l}$ for $i>j$, $l<m$ with $a_{il}\not\equiv 0$, and $\lowsmallsup{D}{[i,l]}\coloneqq\rowexop{A}{i}{m}{k}{l}$ for $i<m$, $l>k$ with $a_{il}\not\equiv 0$.

  Suppose that $i>j$, $l<m$ are such that $a_{il}\not\equiv 0$, and consider the matrix $\lowsmallsup{B}{[i,l]}$. If $i=m$, then $\smallsup{b}[mm]{[m,l]}\neq 0$ and so $h[\lowsmallsup{B}{[m,l]}]=0$ by \Cref{lem.bottom.right.entry}. On the other hand, if $i<m$ then $\smallsup{b}[im]{[i,l]},\smallsup{b}[mk]{[i,l]}\neq 0$, and notice also that $\lowsmallsup{B}{[i,l]}<_{C}A$ by \Cref{rem.order}. Therefore, by minimality of $A$, we have that $h[\lowsmallsup{B}{[i,l]}]=0$. Similarly, one may show that $h[\lowsmallsup{D}{[i,l]}]=0$ for $i<m$, $l>k$ with $a_{il}\not\equiv 0$, and so we deduce that $h[A]=0$.
\end{proof}

\begin{Definition}\label{def.row.col.set}
  We define the sets:
    \begin{enumerate}[label=(\roman*), font=\normalfont, ref=(\roman*)]
      \item\label{def.row.col.set.item.1}$\mathcal{TR}\coloneqq\{A\in\mathcal{T}\mid a_{i1}=1\settext{for}1\leq i<m,\settext{and}a_{mk}=0\settext{for}1<k\leq m\}$.
      \item\label{def.row.col.set.item.2}$\mathcal{TC}\coloneqq\{A\in\mathcal{T}\mid a_{1k}=1\settext{for}1\leq k<m,\settext{and}a_{im}=0\settext{for}1<i\leq m\}$.
    \end{enumerate}
\end{Definition}

\begin{Lemma}\label{lem.rim}
  Let $A\in\mathcal{T}$ and suppose that $A\not\in\mathcal{TR}\cup\mathcal{TC}$. Then $h[A]=0$.
\end{Lemma}
\begin{proof}
  By \Cref{lem.bottom.right.entry} we may assume that $a_{mm}=0$. Suppose that $a_{mk}\neq 0$ for some $k$ with $1<k<m$. Then, by \Cref{lem.outside.rim}, we may assume that $a_{jm}=0$ for $1<j<m$. But then $a_{1m}=b$ and so $\sum_{l<m}a_{1l}=m-1$. Since $A\not\in\mathcal{TC}$ we deduce that there exists some $1\leq l<m$ with $a_{1l}=0$. Now, the relation $C_{l,m}^{1}(A)$ gives that $h[A]=\sum_{j>1}a_{jl}h[\lowsmallsup{B}{[j]}]$ where $\lowsmallsup{B}{[j]}\coloneqq\colexop{A}{l}{m}{1}{j}$ for $j>1$ with $a_{jl}\not\equiv 0$. Suppose that $j>1$ is such that $a_{jl}\not\equiv 0$. If $j=m$ then $\smallsup{b}[mm]{[m]}\neq 0$ and so $h[\lowsmallsup{B}{[m]}]=0$ by \Cref{lem.bottom.right.entry}. Moreover, for $1<j<m$ we have that $\smallsup{b}[mk]{[j]},\smallsup{b}[jm]{[j]}\neq 0$ and so $h[\lowsmallsup{B}{[j]}]=0$ by \Cref{lem.outside.rim}. Therefore, we deduce that $h[A]=0$.

  Hence, we may assume that $a_{mk}=0$ for all $1<k\leq m$ and so it follows that \mbox{$a_{m1}=a-m+1$} and that $\sum_{j<m}a_{j1}=m-1$. However, since $A\not\in\mathcal{TR}$ we must have that $a_{j1}=0$ for some $j$ with $1\leq j<m$. Now, the relation $R_{j,m}^{1}(A)$ gives $h[A]=\sum_{l>1}a_{jl}h[\lowsmallsup{D}{[l]}]$ where $\lowsmallsup{D}{[l]}\coloneqq\rowexop{A}{j}{m}{1}{l}$ for $l>1$ with $a_{jl}\not\equiv 0$. Suppose that $l>1$ is such that $a_{jl}\not\equiv 0$. If $l=m$, then $\smallsup{d}[mm]{[m]}\neq 0$ and so $h[\lowsmallsup{D}{[m]}]=0$ by \Cref{lem.bottom.right.entry}. On the other hand, if $1<l<m$ then $\smallsup{d}[\smash[t]{ml}]{[l]}\neq 0$. Now, if $\smallsup{d}[um]{[l]}\neq 0$ for some $1<u<m$, then $h[\lowsmallsup{D}{[l]}]=0$ by \Cref{lem.outside.rim}. Hence, we may assume that $\smallsup{d}[um]{[l]}=0$ for all $1<u<m$ and so we deduce that $\smallsup{d}[1m]{[l]}=a_{1m}=b$. Since $A\not\in\mathcal{TC}$ we have that there exists some $1\leq k<m$ with $a_{1k}=0$ and hence $\smallsup{d}[1k]{[l]}=0$. Then, the relation $C_{k,m}^{1}(\lowsmallsup{D}{[l]})$ expresses $h[\lowsmallsup{D}{[l]}]$ as a linear combination of $h[F]$s where either $f_{mm}\neq 0$, or $f_{ml}\neq 0$ and $f_{vm}\neq 0$ for some $v$ with $1<v<m$. Once again, \Cref{lem.bottom.right.entry} and \Cref{lem.outside.rim} give that $h[F]=0$ for all such $F$ and so $h[\lowsmallsup{D}{[l]}]=0$. Hence $h[A]=0$.
\end{proof}

\begin{Definition}\label{def.filtr}
  We shall require some additional notation that we shall introduce here:
    \begin{enumerate}[label=(\roman*), font=\normalfont, ref=(\roman*)]
      \item\label{def.filtr.item.1}In order to assist with counting in reverse, set $\tau(i)\coloneqq m-(i-1)$ for $1\leq i\leq m$.
      \item\label{def.filtr.item.2}For $1<i<m$, we define:
        \[ \mathcal{TR}_{i}\coloneqq\{A\in\mathcal{TR}\mid\text{the}\ \tau(j)\text{th-row of}\ A\settext{contains}j\settext{odd entries for}1<j\leq i\}. \]
      \item\label{def.filtr.item.3}For $1<i<m$, we define $\overline{\mathcal{TR}}_{i}\coloneqq\mathcal{TR}_{i}\setminus\mathcal{TR}_{i+1}$, where we set $\mathcal{TR}_{m}\coloneqq\varnothing$.
    \end{enumerate}
\end{Definition}

\begin{Remark}\label{rem.01}
  Let $A\in\mathcal{T}$. Recall that $\sum_{l}a_{\tau(i)l}=i$ for $1<i<m$. Therefore, if $A\in\mathcal{TR}_{i}$ for some $1<i<m$, then the $\tau(j)$th-row of $A$ consists entirely of ones and zeros for all $1<j\leq i$.
\end{Remark}

\begin{Definition}\label{def.ind.tool}
  Let $1<i<m$ and $A\in\overline{\mathcal{TR}}_{i}$. Then:
    \begin{enumerate}[label=(\roman*), font=\normalfont, ref=(\roman*)]
      \item\label{def.ind.tool.item.1}We set $\mathcal{K}_{A}\coloneqq\{2\leq k\leq i\mid a_{uk}=1\settext{for}\tau(i)\leq u\leq\tau(k)\}$.
      \item\label{def.ind.tool.item.2}We set $k_{A}\coloneqq\min\{2\leq k\leq i+1\mid k\not\in\mathcal{K}_{A}\}$.
      \item\label{def.ind.tool.item.3}If $k_{A}\leq i$, we set $j_{A}\coloneqq\min\{k_{A}\leq j\leq i\mid a_{\tau(j)k_{A}}=0\}$.
      \item\label{def.ind.tool.item.4}If $k_{A}\leq i$ and $k_{A}\leq j\leq i$, we denote by $w^{j}(A)\coloneqq(w_{1}^{j}(A),w_{2}^{j}(A),\ldots)$ the decreasing sequence of column-indices within the final $\tau(k_{A})$ columns of $A$ that satisfy $a_{\tau(j)w^{j}_{s}(A)}=1$ for $s\geq 1$.
      \end{enumerate}
\end{Definition}

Notice that the sequence $w^{j}(A)$ has $j-k_{A}+1$ terms.

\begin{Example}\label{exam.tab1}
  We have $k_{A}=4$, $j_{A}=4$, and $w^{5}(A)=(7,5)$, where:
    \[
      A\coloneqq
        \begin{tabular}{|c|c|c|c|c|c|c|c|c|c|c|}
          \hline
            $1$ & $\cdot$ & $\cdot$ & $\cdot$ & $\cdot$ & $\cdot$ & $\cdot$ & $\cdot$ & $\cdot$ \\
          \hline
            $\vdots$ & $\vdots$ & $\vdots$ & $\vdots$ & $\vdots$ & $\vdots$ & $\vdots$ & $\vdots$ & $\vdots$ \\
          \hline
            $1$ & $1$ & $1$ & $1$ & $1$ & $2$ & $0$ & $0$ & $0$ \\
          \hline
            $1$ & $1$ & $1$ & $0$ & $1$ & $0$ & $1$ & $0$ & $0$ \\
          \hline
            $1$ & $1$ & $1$ & $0$ & $1$ & $0$ & $0$ & $0$ & $0$ \\
          \hline
            $1$ & $1$ & $1$ & $0$ & $0$ & $0$ & $0$ & $0$ & $0$ \\
          \hline
            $1$ & $1$ & $0$ & $0$ & $0$ & $0$ & $0$ & $0$ & $0$ \\
          \hline
            $a-m+1$ &0 & $0$ & $0$ & $0$ & $0$ & $0$ & $0$ & $0$ \\
          \hline
        \end{tabular}
      \in\overline{\mathcal{TR}}_{5}.
    \]
\end{Example}

\begin{Lemma}\label{lem.tech.R}
  Let $2<i<m$ and let $A\in\overline{\mathcal{TR}}_{i}$ with $k_{A}\leq i$. Suppose that there exists some index $k$ with $k_{A}<k\leq i$ such that $w^{j}_{t}(A)=w^{j-1}_{\smash[t]{t-1}}(A)$ for all $k_{A}<j\leq k$ and all even $t$. Then for $l\geq k_{A}$, $k_{A}\leq j\leq k$, we have $\sum_{u\geq\tau(j)}a_{ul}\equiv 1$ if and only if $l=w^{j}_{s}(A)$ for some odd $s$.
\end{Lemma}
\begin{proof}
  We proceed by induction on $j$. The case $j=k_{A}$ is clear and so we may assume that $j>k_{A}$ and that the claim holds for all smaller values of $j$ in the given range. Let $l\geq k_{A}$ and suppose that $\sum_{u\geq\tau(j)}a_{ul}\equiv 1$. Suppose, for the moment, that $a_{\tau(j)l}=0$. Then $\sum_{\smash[b]{u\geq\tau(j)}}a_{ul}=\sum_{u\geq\tau(j-1)}a_{ul}$, and so $l=w^{\smash[t]{j-1}}_{s}(A)$ for some odd $s$ by the inductive hypothesis. However, $w^{j}_{\smash[t]{s+1}}(A)=w^{j-1}_{s}(A)=l$ and so $a_{\tau(j)l}=1$, contradicting that \mbox{$a_{\tau(j)l}=0$}. Hence, $a_{\tau(j)l}=1$ and so $l=w^{j}_{s}(A)$ for some $s$. Moreover, $\sum_{u\geq\tau(j)}a_{ul}\equiv 1$ if and only if $\sum_{\smash[b]{u\geq\tau(j-1)}}a_{ul}\equiv 0$ and so by the inductive hypothesis $l\neq w^{j-1}_{\smash[t]{s'}}(A)$ for any odd $s'$. Now, if $s$ is even then $w^{j}_{s}(A)=w^{j-1}_{\smash[t]{s-1}}(A)$, leading to a contradiction. Hence, $s$ must be odd. Conversely, suppose that $l=w^{j}_{s}(A)$ for some odd $s$, and suppose, for the sake of contradiction, that $\sum_{\smash[b]{u\geq\tau(j)}}a_{ul}\equiv 0$. Then, there exists some $k_{A}\leq j'<j$ such that $a_{\tau(j')l}=1$, and we choose $j'$ to be maximal with this property. Therefore, $a_{ul}=0$ for $\tau(j)<u<\tau(j')$ and $\sum_{\smash[b]{u\geq\tau(j')}}a_{ul}\equiv 1$. Then, by the inductive hypothesis, $l=w^{j'}_{\smash[t]{s'}}(A)$ for some odd $s'$. But then $w^{j'+1}_{\smash[t]{s'+1}}(A)=w^{j'}_{\smash[t]{s'}}(A)=l$, by our assumption, and so $a_{\tau(j'+1)l}=1$. Now, by the maximality of $j'$, we must have $j'+1=j$. Thus, $l=w^{j'+1}_{\smash[t]{s'+1}}(A)=w^{j}_{\smash[t]{s'+1}}(A)=w^{j}_{s}(A)$ and so $s'+1=s$, which is impossible since $s'$ and $s$ are both odd. Hence $\sum_{u\geq\tau(j)}a_{ul}\equiv 1$, and so we are done.
\end{proof}

\begin{Lemma}\label{lem.tech.R1}
  Let $2<i<m$ and let $A\in\overline{\mathcal{TR}}_{i}$ with $k_{A}\leq i$. Suppose that $z_{\tau(j),l}(A)=0$ for all $k_{A}\leq j\leq i$, $k_{A}\leq l<m$ with $a_{\tau(j)l}=1$. Then $w^{j}_{s}(A)=w^{j-1}_{\smash[t]{s-1}}(A)$ for $k_{A}<j\leq i$ and even $s$ with $s\leq j-k_{A}+1$.
\end{Lemma}
\begin{proof}
  We fix $i$ and we proceed by induction on $j$, with the base case being $j=k_{A}+1$. Here $w_{\smash[t]{1}}^{\tau(j)}(A)=(w_{\smash[t]{1}}^{\tau(j)}(A),w_{\smash[t]{2}}^{\tau(j)}(A))$ and for $w\coloneqq w_{\smash[t]{2}}^{\tau(j)}(A)$ we have $z_{\tau(j),w}(A)=0$. Now, by \myCref{lem.critical.rel.coef}[lem.critical.rel.coef.item.2] we have $z_{\tau(j),w}(A)=\sum_{u>\tau(j)}a_{uw}+\sum_{v>w}a_{\tau(j)v}=a_{\tau(j-1)w}+1$. Therefore, the entry $a_{\tau(j-1)w}$ is odd and so $w=w_{\smash[t]{1}}^{k_{A}}(A)$ as required. Suppose now that $k_{A}+1<j\leq i$ and that the claim holds for smaller values of $j$ in the given range.

  Suppose that $s$ is even and set $l\coloneqq w^{j}_{s}(A)$. Then $\sum_{u>\tau(j)}a_{ul}+s-1\equiv 0$ by \myCref{lem.critical.rel.coef}[lem.critical.rel.coef.item.2] since $z_{\tau(j),l}(A)=0$. Therefore, $\sum_{u\geq\tau(j-1)}a_{ul}\equiv 1$ and so by the inductive hypothesis, \Cref{lem.tech.R} applies and gives that $l=w^{j-1}_{\smash[t]{s'}}(A)$ for some odd $s'$ with $s'\leq j-k_{A}$. Now, the sequence $w^{j}(A)$ has exactly one extra term compared to $w^{j-1}(A)$ and so the number of even indices in $w^{j}(A)$ equals the number of odd indices in $w^{j-1}(A)$. It follows that $s'=s-1$ and so we are done.
\end{proof}

\begin{Lemma}\label{lem.k.lem}
  Let $1<i<m$ and let $A\in\overline{\mathcal{TR}}_{i}$ with $k_{A}\leq i$. Suppose that \mbox{$w_{1}^{j}(A)>w_{1}^{j-1}(A)$} for all $j_{A}<j\leq i$. Then we may express $h[A]$ as a linear combination of $h[B]$s for some $B\in\mathcal{T}$ where either:
    \begin{enumerate}[label=(\roman*), font=\normalfont, ref=(\roman*)]
      \item\label{lem.k.lem.item.1}$B\in\overline{\mathcal{TR}}_{i'}$ for some $i'<i$,
      \item\label{lem.k.lem.item.2}$B\in\overline{\mathcal{TR}}_{i}$ with $k_{B}>k_{A}$,
      \item\label{lem.k.lem.item.3}$B\in\overline{\mathcal{TR}}_{i}$ with $k_{B}=k_{A}$ and $B<_{C}A$, which is witnessed within the final $\tau(w_{1}^{\tiniersub{j}{A}}(A))$ columns of $A$ and $B$.
    \end{enumerate}
  Moreover, if $A\not\in\mathcal{TC}$ then $B\not\in\mathcal{TC}$ for all such $B$ listed above.
\end{Lemma}
\begin{proof}
  To ease notation we set $u\coloneqq\tau(j_{A})>1$, $k\coloneqq k_{A}$, and $w\coloneqq w_{1}^{j_{A}}(A)$. Notice that $w>k$, and that $a_{uk}=0$ and $a_{uw}=1$. The relation $C_{k,w}^{u}(A)$ gives $h[A]=\sum_{l\neq u}a_{lk}h[\lowsmallsup{B}{[l]}]$ where $\lowsmallsup{B}{[l]}\coloneqq\colexop{A}{k}{w}{u}{l}$ for $l\neq u$ with $a_{lk}\not\equiv 0$. Let $l\neq u$ be such that $a_{lk}\not\equiv 0$, and let $\lowsmallsup{k}{[l]}\coloneqq k_{\lowsmallestsup{B}{[l]}}$, $\lowsmallsup{j}{[l]}\coloneqq j_{\lowsmallestsup{B}{[l]}}$, and $\lowsmallsup{w}{[l]}\coloneqq\lowsmallsup{w}[1]{\lowsmallestsup{j}{[l]}}(\lowsmallsup{B}{[l]})$. We shall proceed by induction on $j_{A}$, decreasing from $j_{A}=i$.

  Firstly, suppose that $j_{A}=i$. If $l>u$ and $a_{lw}\neq 0$, then $\lowsmallsup{B}{[l]}\in\overline{\mathcal{TR}}_{i'}$ for some $i'<i$, and so $\lowsmallsup{B}{[l]}$ is as described in \caseCref{lem.k.lem.item.1}. Now, if $l>u$ with $a_{lw}=0$, then $\lowsmallsup{k}{[l]}=k$, $\lowsmallsup{B}{[l]}<_{C}A$, and the final column in which $\lowsmallsup{B}{[l]}$ and $A$ differ is the $w$th-column. Hence, here $\lowsmallsup{B}{[l]}$ is as described in \caseCref{lem.k.lem.item.3}. On the other hand, if $l<u$, then $\lowsmallsup{k}{[l]}>k$ and $\lowsmallsup{B}{[l]}$ is as described in \caseCref{lem.k.lem.item.2}.

  Now, suppose that $j_{A}<i$ and that the claim holds for all $D\in\overline{\mathcal{TR}}_{i}$ with $j_{A}<j_{D}\leq i$. We split our consideration into steps:

  \underline{Step 1}: If $l>u$ and $a_{lw}\neq 0$, then $\lowsmallsup{B}{[l]}\in\overline{\mathcal{TR}}_{i'}$ for some $i'<i$, and so $\lowsmallsup{B}{[l]}$ is as described in \caseCref{lem.k.lem.item.1}. On the other hand, if $l>u$ and $a_{lw}=0$, then $\lowsmallsup{B}{[l]}\in\overline{\mathcal{TR}}_{i}$ with $\lowsmallsup{k}{[l]}=k$ and $\lowsmallsup{B}{[l]}<_{C}A$. Moreover, the final column in which $\lowsmallsup{B}{[l]}$ and $A$ differ in this case is the $w$th-column and so $\lowsmallsup{B}{[l]}$ is as described in \caseCref{lem.k.lem.item.3}.

  \underline{Step 2}: Now, if $\tau(i)\leq l<u$ with $a_{lw}\neq 0$. Then $\lowsmallsup{B}{[l]}\in\overline{\mathcal{TR}}_{m-l}$ with $m-l<i$ since $l\geq\tau(i)=m-i+1$, and so $\lowsmallsup{B}{[l]}$ is as described as in \caseCref{lem.k.lem.item.1}.

  \underline{Step 3}: On the other hand, if $\tau(i)\leq l<u$ and $a_{lw}=0$, then $\lowsmallsup{B}{[l]}\in\overline{\mathcal{TR}}_{i}$ with $\lowsmallsup{k}{[l]}=k$ and $\lowsmallsup{j}{[l]}>j_{A}$. Moreover, the final column in which $A$ and $B$ differ is the $w$th-column, and so $w_{1}^{j}(\lowsmallsup{B}{[l]})=w_{1}^{j}(A)$ for all $j_{A}<j\leq i$, and so in particular $w_{1}^{j}(\lowsmallsup{B}{[l]})>w_{1}^{j-1}(\lowsmallsup{B}{[l]})$ for each $\lowsmallsup{j}{[l]}<j\leq i$. Hence, by the inductive hypothesis, $\lowsmallsup{B}{[l]}$ must satisfy the claim, and so $h[\lowsmallsup{B}{[l]}]$ may be written as a linear combination of $h[D]$s for some $D\in\mathcal{T}$ where either:
    \begin{enumerate}[label=(\roman*), font=\normalfont, ref=(\roman*)]\setcounter{enumi}{3}
      \item\label{lem.k.lem.item.4}$D\in\overline{\mathcal{TR}}_{i'}$ for some $i'<i$,
      \item\label{lem.k.lem.item.5}$D\in\overline{\mathcal{TR}}_{i}$ with $k_{D}>\lowsmallsup{k}{[l]}$,
      \item\label{lem.k.lem.item.6}$D\in\overline{\mathcal{TR}}_{i}$ with $k_{D}=\lowsmallsup{k}{[l]}$ and $D<_{C}\lowsmallsup{B}{[l]}$, which is witnessed within the final $\tau(\lowsmallsup{w}{[l]})$ columns of $\lowsmallsup{B}{[l]}$ and $D$.
    \end{enumerate}
  Any such $D$ as in \labelcref{lem.k.lem.item.4} is as described in \caseCref{lem.k.lem.item.1}, whereas any such $D$ as in \labelcref{lem.k.lem.item.5} is as described in \caseCref{lem.k.lem.item.2} since $\lowsmallsup{k}{[l]}=k_{A}$. Now, notice that the final $\tau(\lowsmallsup{w}{[l]})$ columns of $A$ and $\lowsmallsup{B}{[l]}$ match since $\lowsmallsup{w}{[l]}>w$, and so any such $D$ as in \labelcref{lem.k.lem.item.6} also satisfies $D<_{C}A$ (witnessed within the final $\tau(w)$ columns of $A$ and $D$), and so is as described in \caseCref{lem.k.lem.item.3}.

  \underline{Step 4}: Finally, if $l<\tau(i)$, then $\lowsmallsup{B}{[l]}\in\overline{\mathcal{TR}}_{i}$. Moreover, if $a_{tk}=1$ for all $\tau(i)\leq t<\tau(j_{A})$, then $\lowsmallsup{k}{[l]}>k$ and so $\lowsmallsup{B}{[l]}$ is as described in \caseCref{lem.k.lem.item.2}. On the other hand, if $a_{tk}=0$ for some $t$ in this range, then $\lowsmallsup{k}{[l]}=k$ with $\lowsmallsup{j}{[l]}>j_{A}$ and then one may proceed as in Step 3 above.

  Now, suppose that $A\not\in\mathcal{TC}$ but $\lowsmallsup{B}{[l]}\in\mathcal{TC}$ for some $l\neq u$ with $a_{lk}\not\equiv 0$. Notice that this forces $l=1$ and $a_{lk}=2${,} which contradicts that $a_{lk}\not\equiv 0$. Hence if $A\not\in\mathcal{TC}$, then $\lowsmallsup{B}{[l]}\not\in\mathcal{TC}$ for all $l\neq u$ with $a_{lk}\not\equiv 0$. By applying this argument recursively, it follows that if $A\not\in\mathcal{TC}$, then all such $B$ produced by this procedure satisfy $B\not\in\mathcal{TC}$ as well.
\end{proof}

\begin{Lemma}\label{lem.k.lem2}
  Let $1<i<m-1$ and let $A\in\overline{\mathcal{TR}}_{i}$ with $k_{A}=i+1$. Then we may express $h[A]$ as a linear combination of $h[B]$s for some $B\in\mathcal{T}$ where either:
    \begin{enumerate}[label=(\roman*), font=\normalfont, ref=(\roman*)]
      \item\label{lem.k.lem2.item.1}$B\in\overline{\mathcal{TR}}_{i'}$ for some $i'<i$,
      \item\label{lem.k.lem2.item.2}$B\not\in\mathcal{TR}$.
    \end{enumerate}
  Moreover, if $A\not\in\mathcal{TC}$ then $B\not\in\mathcal{TC}$ for all such $B$ listed above.
\end{Lemma}
\begin{proof}
  Firstly, recall that the sum of the entries in the $\tau(i+1)$th-row of $A$ is $i+1$. Now, since $A\not\in\mathcal{TR}_{i+1}$, we deduce that the $\tau(i+1)$th-row of $A$ contains at most $i-1$ odd entries. Hence, there exists some $1<s\leq i$ such that $a_{\tau(i+1)s}$ is even and we choose $s$ be minimal with this property. To ease notation, we set $q\coloneqq\tau(i+1)$ and $u\coloneqq\tau(s)$. Note that $a_{us}=1$. The relation $R_{q,u}^{s}(A)$ gives that $h[A]=\sum_{l\neq s}a_{ql}h[\lowsmallsup{B}{[l]}]$ where $\lowsmallsup{B}{[l]}\coloneqq\rowexop{A}{q}{u}{s}{l}$ for $l\neq s$ with $a_{ql}\not\equiv 0$.

  If $l=1$, then $\lowsmallsup{B}{[1]}\not\in\mathcal{TR}$, and so $\lowsmallsup{B}{[1]}$ is as described in \caseCref{lem.k.lem2.item.2}. Now, if $1<l<s$, then $\lowsmallsup{B}{[l]}\in\overline{\mathcal{TR}}_{s-1}$ with $s-1<i$, and so $\lowsmallsup{B}{[l]}$ is as described in \caseCref{lem.k.lem2.item.1}. Meanwhile, if $l>s$, then $\lowsmallsup{B}{[l]}\in\overline{\mathcal{TR}}_{i}$, and there exists some $s<t\leq i$ (depending on $l$) such that $\smallsup{b}[qt]{[l]}$ is even, and we take $t$ to be minimal with this property. The relation $R_{q,\tau(t)}^{t}(\lowsmallsup{B}{[l]})$ expresses $h[\lowsmallsup{B}{[l]}]$ as a linear combination of $h[D]$s for some $D\in\mathcal{T}$ that must either fit into one of the cases described in the statement of the claim, or otherwise once again $D\in\overline{\mathcal{TR}}_{i}$ and there exists some $t<v\leq i$ such that $d_{qv}$ is even, and we take $v$ to be minimal with this property. Noting that $v>t>s$, it is clear that this process must terminate, hence providing the desired expression for $h[A]$.

  Now, suppose that $A\not\in\mathcal{TC}$ but $\lowsmallsup{B}{[l]}\in\mathcal{TC}$ for some $l\neq s$ with $a_{ql}\not\equiv 0$. Then, notice that $\lowsmallsup{B}{[l]}$ agrees with $A$ outside the $\tau(i+1)$th and $\tau(s)$th rows, and so in particular they agree in the first row since $i<m-1$. Hence $a_{1v}=\lowsmallsup{b}[1v]{[l]}=1$ for $1\leq v<m$ since $\lowsmallsup{B}{[l]}\in\mathcal{TC}$. Now, by considering the first row-sum and the last column-sum of $A$, we deduce that $a_{1m}=b$ and $a_{vm}=0$ for $1<v\leq m$. However, this implies that $A\in\mathcal{TC}$, which is a contradiction. Once again, by applying this argument recursively, it follows that if $A\not\in\mathcal{TC}$, then all such $B$ produced by this procedure satisfy $B\not\in\mathcal{TC}$ as well.
\end{proof}

\begin{Lemma}\label{lem.sruct.1}
  Let $1<i<m-1$ and let $A\in\overline{\mathcal{TR}}_{i}$. Then we may express $h[A]$ as a linear combination of $h[B]$s for some $B\in\mathcal{T}\setminus\mathcal{TR}$. Moreover, if $A\not\in\mathcal{TC}$ then all such $B$ satisfy $B\not\in\mathcal{TR}\cup\mathcal{TC}$.
\end{Lemma}
\begin{proof}
  We proceed by induction on $i\geq 2$. Firstly, suppose that $i=2$. Since $A\in\mathcal{T}\setminus\mathcal{TR}_{3}$ with $\sum_{l}a_{\tau(3)l}=3$, the $(m-2)$th-row of $A$ must consist of a single odd entry, which must then be equal to $1$, and be located in the first column of $A$. On the other hand, since $A\in\mathcal{TR}_{2}$, there exists a unique $l>1$ with $a_{(m-1)l}=1$. The relation $R_{m-2,m-1}^{l}(A)$ gives $h[A]=h[B]$ for $B\coloneqq\rowexop{A}{m-2}{m-1}{l}{1}$. Evidently, $B\not\in\mathcal{TR}$, and so the claim holds for $i=2$.

  Now, we suppose that $i>2$ and that the claim holds for all $B\in\mathcal{T}$ such that $B\in\overline{\mathcal{TR}}_{i'}$ for some $2\leq i'<i$. Suppose, for the sake of contradiction, that the claim fails for this particular value of $i$ and consider the set of counterexamples $A\in\overline{\mathcal{TR}}_{i}$ whose value of $k_{A}$ is maximal amongst all counterexamples. Now, we choose $A$ to be the element of this set that is minimal with respect to the column-ordering. In other words, if $D\in\overline{\mathcal{TR}}_{i}$ is a counterexample to the claim, then either $k_{D}<k_{A}$, or $k_{D}=k_{A}$ and $D\geq_{C}A$.

  Now if $k_{A}=i+1$, then \Cref{lem.k.lem2} states that we may express $h[A]$ as a linear combination of some $h[B]$s for some $B\in\mathcal{T}$ where either $B\in\overline{\mathcal{TR}}_{i}$ with $i'<i$, or $B\not\in\mathcal{TR}$. In the first case the inductive hypothesis states that $h[B]$ can be expressed as a linear combination of some $h[D]$s with $D\not\in\mathcal{TR}$, whilst in the second case we have $B\in\mathcal{T}\setminus\mathcal{TR}$. Thus, $h[A]$ satisfies the statement of the claim which contradicts that $A$ was chosen to be a counterexample.

  Hence, we may assume that $k_{A}\leq i$. Suppose, for the sake of contradiction, that there exists \mbox{$k_{A}\leq j\leq i$}, $k_{A}\leq k<m$ such that $a_{\tau(j)k}=1$ and $z_{\tau(j),k}(A)=1$. The relation $Z_{\tau(j),k}(A)$ gives the expression:
    \begin{equation}\label{eq.lem.sruct.1.1}
      h[A]=\sum_{\substack{u<\tau(j)\\l>k}}a_{ul}h[\lowsmallsup{B}{[u,l]}]+\sum_{\substack{u>\tau(j)\\l<k}}a_{ul}h[\lowsmallsup{B}{[u,l]}],
    \end{equation}
  where $\lowsmallsup{B}{[u,l]}\coloneqq\rowexop{A}{u}{\tau(j)}{k}{l}$ for all such $(u,l)$ satisfying $a_{ul}\not\equiv 0$.

  Now, set $B\coloneqq\lowsmallsup{B}{[u,l]}$ where $(u,l)$ is as in \labelcref{eq.lem.sruct.1.1} with $a_{ul}\not\equiv 0$. We claim that $B$ fits into one of the following cases: $B\not\in\mathcal{TR}$, $B\in\mathcal{TR}_{i'}$ for some $i'<i$, or $B\in\overline{\mathcal{TR}}_{i}$ with $k_{B}=k_{A}$ and $B<_{C}A$. We provide full details for the case where $u>\tau(j)$, $l<k$ with the other case, that is $u<\tau(j)$, $l>k$, being similar.

  If $l=1$ then $B\not\in\mathcal{TR}$ and so $B$ is of the desired form. Now, if $1<l<k_{A}$, then either $u\geq\tau(k_{A})$ or $\tau(j)<u<\tau(k_{A})$. In the first case, we have $B\in\overline{\mathcal{TR}}_{j-1}$, whilst in the second case we have $B\in\overline{\mathcal{TR}}_{\tau(u-1)}$ if $a_{uk}=1$ and $B\in\overline{\mathcal{TR}}_{j-1}$ if $a_{uk}=0$. Hence, in either case, we deduce that $B\in\overline{\mathcal{TR}}_{i'}$ for some $i'<i$. Suppose now that $k_{A}\leq l<k$, then we must have $\tau(j)<u\leq\tau(k_{A})$ since $a_{ul}\not\equiv 0$. Now, if $a_{uk}=1$ then $B\in\overline{\mathcal{TR}}_{\tau(u-1)}$, whilst if $a_{uk}=0$ and $a_{\tau(j)l}=1$, then $B\in\overline{\mathcal{TR}}_{j-1}$. Finally, if $a_{uk}=0$ and $a_{\tau(j)l}=0$, then $B\in\overline{\mathcal{TR}}_{i}$ with $k_{B}=k_{A}$ and $B<_{C}A$. But then, either by the inductive hypothesis on $i$, or by the minimality of $A$, all such $B$ produced in this procedure must satisfy the statement of the claim, and hence so must $A$, which contradicts that $A$ was chosen to be a counterexample.

  Therefore, we may assume that that $z_{\tau(j),k}(A)=0$ for all $k_{A}\leq j\leq i$, $k_{A}\leq k<m$ such that $a_{\tau(j)k}=1$. Then, by \Cref{lem.tech.R1} and \Cref{lem.k.lem}, we may express $h[A]$ as a linear combination of $h[B]$s for some $B\in\mathcal{T}$ where either: $B\in\overline{\mathcal{TR}}_{i'}$ for some $i'<i$, $B\in\overline{\mathcal{TR}}_{i}$ with $k_{B}>k_{A}$, or $B\in\overline{\mathcal{TR}}_{i}$ with $k_{B}=k_{A}$ and $B<_{C}A$. But then, either by the inductive hypothesis on $i$, maximality of $k_{A}$, or minimality of $A$, each such $B$ must satisfy the statement of the claim, and hence so must $A$, which contradicts that $A$ was chosen to be a counterexample. Thus, no such counterexample may exist. Finally, once again, it is clear to see from the steps taken above that if $A\not\in\mathcal{TC}$, then all such $B$ produced by this procedure satisfy $B\not\in\mathcal{TC}$ as well.
\end{proof}

\begin{Corollary}\label{cor.sruct.1}
  Let $2\leq i<m-1$ and let $A\in\overline{\mathcal{TR}}_{i}$ with $A\not\in\mathcal{TC}$. Then $h[A]=0$.
\end{Corollary}
\begin{proof}
  By \Cref{lem.sruct.1}, we may express $h[A]$ as a linear combination of $h[B]$s for some $B\in\mathcal{T}$ with $B\not\in\mathcal{TR}\cup\mathcal{TC}$. But $h[B]=0$ for all such $B$ by \Cref{lem.rim}, and so the result follows.
\end{proof}

\begin{Lemma}\label{lem.sruct.2}
  Let $A\in\mathcal{TR}\setminus\mathcal{TC}$. Then $h[A]=0$.
\end{Lemma}
\begin{proof}
  Suppose, for the sake of contradiction, that the claim is false, and let $A\in\mathcal{T}$ be a counterexample that is minimal with respect to the column-ordering of \myCref{def.order}[def.order.item.2]. By \Cref{cor.sruct.1}, we may assume that $A\not\in\overline{\mathcal{TR}}_{i}$ for any $i<m-1$, and so we must have that $A\in\mathcal{TR}_{m-1}\setminus\mathcal{TC}$ since $A\in\mathcal{TR}$. Hence, for each $1<u<m$, either $a_{um}=0$ or $a_{um}=1$, and we claim that there exists at least one $u$ in this range with $a_{um}=1$. Indeed, suppose otherwise, then there exists some $1<v<m$ with $a_{1v}$ even since $A\not\in\mathcal{TC}$. But then the relation $C_{vm}^{1}(A)$ expresses $h[A]$ as a linear combination of $h[B]$s for some $B\in\mathcal{T}$ with $B<_{C}A$ and \mbox{$B\in\mathcal{TR}\setminus\mathcal{TC}$}. But $h[B]=0$ for all such $B$ by minimality of $A$, which contradicts that $A$ was chosen to be a counterexample. We hence write $(u_{1},\ldots,u_{s})$ for the increasing sequence whose terms are given by all such $u$. Firstly, suppose that $s>1$ and set $u\coloneqq u_{s-1}$ and $u'\coloneqq u_{s}$. By \myCref{lem.critical.rel.coef}[lem.critical.rel.coef.item.3], we have that $z_{u,m}(A)+z_{u',m}(A)=1$ and so the relation $Z_{u,m}(A)+Z_{u',m}(A)$ is given by:
    \begin{equation}\label{eq.lem.sruct.2.1}
      h[A]=\sum_{\substack{v>u\\l<m}}a_{vl}h[\lowsmallsup{B}{[v,l]}]+\sum_{\substack{v>u\\l<m}}a_{vl}h[\lowsmallsup{D}{[v,l]}],
    \end{equation}
  where $\lowsmallsup{B}{[v,l]}\coloneqq\rowexop{A}{v}{u}{m}{l}$ and $\lowsmallsup{D}{[v,l]}\coloneqq\rowexop{A}{v}{u'}{m}{l}$ for all such $(v,l)$ with $a_{vl}\not\equiv 0$. Now, let $(v,l)$ be as in \labelcref{eq.lem.sruct.2.1} with $a_{vl}\not\equiv 0$.

  If $l=1$, then \mbox{$\lowsmallsup{B}{[v,1]},\lowsmallsup{D}{[v,1]}\not\in\mathcal{TR}\cup\mathcal{TC}$} and so $h[\lowsmallsup{B}{[v,1]}]=h[\lowsmallsup{D}{[v,1]}]=0$ by \Cref{lem.rim}. On the other hand, if $l>1$, then $\lowsmallsup{B}{[v,l]},\lowsmallsup{D}{[v,l]}\in\mathcal{TR}\setminus\mathcal{TC}$ and $A<_{C}\lowsmallsup{B}{[v,l]},\lowsmallsup{D}{[v,l]}$. Hence, by the minimality of $A$, once again we deduce that $h[\lowsmallsup{B}{[v,l]}]=h[\lowsmallsup{D}{[v,l]}]=0$. Thus $h[A]=0$, which contradicts that $A$ was chosen to be a counterexample.

  Hence we may assume that $s=1$, or in other words that there exists a unique $u$ in the range $1<u<m$ such that $a_{um}=1$, and so then $z_{1,m}(A)=1$ by \myCref{lem.critical.rel.coef}[lem.critical.rel.coef.item.5]. By applying similar considerations to the above to the relation $Z_{1,m}(A)$, we once again reach a contradiction, and so no such counterexample may exist.
\end{proof}

\begin{Definition}\label{def.tc}
  For $1<i<m$, similarly to $\mathcal{TR}_{i}$ of \myCref{def.filtr}[def.filtr.item.2], we define:
    \[ \mathcal{TC}_{i}\coloneqq\{A\in\mathcal{TC}\mid\text{the}\ \tau(j)\text{th-column of}\ A\settext{contains}j\settext{odd entries for}1<j\leq i\}. \]
\end{Definition}

\begin{Remark}\label{rem.tran.flat}
  Firstly, note that by \Cref{lem.hom.rel}, we see that the transpose homomorphism $h'\in\Hom_{\k\mathfrak{S}_{r}}(M(\beta),M(\alpha))$ of $h$ is relevant. Now, the results proven above are independent of the values of $a$ and $b$, provided that they satisfy the parity condition \mbox{$a-m\equiv b$}. In particular, note that this condition is preserved under the swap $(a,b)\leftrightarrow(a',b')$, where $a'\coloneqq b+m-1$, $b'\coloneqq a-m+1$. But, as in \Cref{rem.swap}, this swap is equivalent to the swap $\lambda\leftrightarrow\lambda'$ and accordingly $\alpha\leftrightarrow\beta$ and $\mathcal{T}\leftrightarrow\mathcal{T}'$. Therefore, by defining the subsets $\mathcal{TR}',\mathcal{TC}'\subseteq\mathcal{T}'$ analogously to $\mathcal{TR},\mathcal{TC}\subseteq\mathcal{T}$, we obtain the analogous results to those shown in this section for the coefficients $h'[A']$ of the $\rho[A']$ in $h'$.
\end{Remark}

\begin{Proposition}\label{prop.sruct.3}
  Let $A\in\mathcal{T}$ and suppose that $A\not\in\mathcal{TR}_{m-1}\cap\mathcal{TC}_{m-1}$. Then $h[A]=0$.
\end{Proposition}
\begin{proof}
  Suppose that $D\in\mathcal{T}$ is such that $h[D]\neq 0$. Then, we may assume that we have \mbox{$D\in\mathcal{TR}\cup\mathcal{TC}$} since otherwise $h[D]=0$ by \Cref{lem.rim}. Moreover, we may assume that \mbox{$D\not\in\mathcal{TR}\setminus\mathcal{TC}$} since otherwise $h[D]=0$ by \Cref{lem.sruct.2}. On the other hand, if \mbox{$D\in\mathcal{TC}\setminus\mathcal{TR}$}, then \mbox{$D'\in\mathcal{TR}'\setminus\mathcal{TC}'$}, where $\mathcal{TR}',\mathcal{TC}'\subseteq\mathcal{T}'$ are as defined in \Cref{rem.tran.flat}. But then we have \mbox{$h[D]=h'[D']=0$} \'a la \Cref{lem.sruct.2}, which contradicts our choice of $D$, and so we may assume that $D\not\in\mathcal{TC}\setminus\mathcal{TR}$. In sum, we have shown that $h[D]=0$ for all $D\in\mathcal{T}$ with $D\not\in\mathcal{TR}\cap\mathcal{TC}$. In particular, to prove the \nameCref{prop.sruct.3}, we may assume that $A\in\mathcal{TR}\cap\mathcal{TC}$. Now, if $A\not\in\mathcal{TR}_{m-1}$, then there exists some $i$ with $1<i<m-1$ such that $A\in\overline{\mathcal{TR}}_{i}$. But then \Cref{lem.sruct.1} allows one to express $h[A]$ as a linear combination of $h[B]$s for some $B\in\mathcal{T}$ with $B\not\in\mathcal{TR}$. But then every such $B$ satisfies $B\not\in\mathcal{TR}\cap\mathcal{TC}$ and hence that $h[B]=0$ as shown above, and so $h[A]=0$. On the other hand, if $A\not\in\mathcal{TC}_{m-1}$, then $A'\not\in\mathcal{TR}'_{m-1}$ where $\mathcal{TR}'_{m-1}\subseteq\mathcal{T}'$ is defined analogously to $\mathcal{TR}_{m-1}\subseteq\mathcal{T}$. But then $h[A]=h'[A']=0$ by the $'$-decorated analogue to the argument outlined above, and so we are done.
\end{proof}

\begin{Theorem}\label{the.end.res}
  Let $\lambda=(a,m-1,\ldots,2,1^{b})$ with $a\geq m\geq 2$, $b\geq 1$, where $r\coloneqq\deg(\lambda)$, and suppose that the parameters $a$, $b$, and $m$ satisfy the parity condition: $a-m\equiv b\mod{2}$. Then $\End_{\k\mathfrak{S}_{r}}(\Sp(\lambda))\cong\k$.
\end{Theorem}
\begin{proof}
  Let $\bar{h}$ be a non-zero endomorphism of $\Sp(\lambda)$, which we identify with a relevant homomorphism $h\in\Hom_{\k\mathfrak{S}_{r}}(M(\alpha),M(\beta))$ as in \Cref{rem.ide.end}. If $A\in\mathcal{T}$ with $h[A]\neq 0$, then $A\in\mathcal{TR}_{m-1}\cap\mathcal{TC}_{m-1}$ by \Cref{prop.sruct.3}. But since $\sum_{v}a_{\tau(i)v}=i$, $\sum_{u}a_{u\tau(j)}=j$ for $1<i,j<m$, this set consists solely of the matrix:
    \[
      A_{0}\coloneqq
        \begin{tabular}{|c|c|c|c|c|c|c|}
          \hline
            $1$ & $1$ & $1$ & $\dots$ & $1$ & $1$ & $b$ \\
          \hline
            $1$ & $1$ & $1$ & $\dots$ & $1$ & $1$ & $0$ \\
          \hline
            $1$ & $1$ & $1$ & $\dots$ & $1$ & $0$ & $0$ \\
          \hline
            $\vdots$ & $\vdots$ & $\vdots$ & $\ddots$ & $\vdots$ & $\vdots$ & $\vdots$ \\
          \hline
            $1$ & $1$ & $1$ & $\dots$ & $0$ & $0$ & $0$ \\
          \hline
            $1$ & $1$ & $0$ & $\dots$ & $0$ & $0$ & $0$ \\
          \hline
            $a-m+1$ & $0$ & $0$ & $\dots$ & $0$ & $0$ & $0$ \\
          \hline
        \end{tabular}
      \ .
    \]
  Therefore, we have $h=h[A_{0}]\rho[A_{0}]$, and so we are done.
\end{proof}

\end{document}